\documentclass[12pt]{amsart}
\usepackage{amsmath,amssymb,amsthm,latexsym,amsbsy,amsfonts}
\usepackage{color,verbatim}
\usepackage{pstricks}
\usepackage{tabularx}
\usepackage{rotating}
\usepackage{graphicx,epsfig}

\newcommand{\C}{{\mathbb C}}
\newcommand{\D}{{\mathbb D}}

\newcommand{\R}{{\mathbb R}}

\newcommand{\mP}{\mathbb P}

\newcommand{\mE}{\mathbb E}
\newcommand{\mH}{\mathbb H}
\newcommand{\mL}{\mathbb L}
\newcommand{\mT}{\mathbb T}

\newcommand{\mV}{\mathbb V}
\newcommand{\mW}{\mathbb W}

\newcommand{\inpro}[2]{\left\langle{#1},{#2}\right\rangle}

\newcommand{\norm}[2]{\|{#1}\|_{#2}}

\newcommand{\bnorm}[2]{\Big\|{#1}\Big\|_{#2}}
\newcommand{\snorm}[2]{\left|{#1}\right|_{#2}}

\newcommand{\al}{\alpha}

\newcommand{\vecH}{\boldsymbol{H}}

\newcommand{\vecM}{\boldsymbol{M}}

\newcommand{\veca}{\boldsymbol{a}}
\newcommand{\vecb}{\boldsymbol{b}}
\newcommand{\vecc}{\boldsymbol{c}}

\newcommand{\vecg}{\boldsymbol{g}}

\newcommand{\vecm}{\boldsymbol{m}}
\newcommand{\vecn}{\boldsymbol{n}}

\newcommand{\vecu}{\boldsymbol{u}}
\newcommand{\vecv}{\boldsymbol{v}}
\newcommand{\vecw}{\boldsymbol{w}}
\newcommand{\vecx}{\boldsymbol{x}}

\newcommand{\vecvarphi}{\boldsymbol{\varphi}}

\newcommand{\vecpsi}{\boldsymbol{\psi}}
\newcommand{\vecxi}{\boldsymbol{\xi}}
\newcommand{\veczeta}{\boldsymbol{\zeta}}

\DeclareMathOperator*{\esssup}{ess\,sup \/}

\newcommand{\cF}{{\mathcal F}}

\newcommand{\cI}{{\mathcal I}}

\newcommand{\cL}{{\mathcal L}}
\newcommand{\cM}{{\mathcal M}}
\newcommand{\cN}{{\mathcal N}}

\newcommand{\cT}{{\mathcal T}}

\newcommand{\goto}{\rightarrow}
\newcommand{\gotoo}{\longrightarrow}

\newcommand{\p}{\partial}
\newcommand{\wtd}{\widetilde}

\newcommand{\pa}{\partial}

\newcommand{\ds}{\, ds}
\newcommand{\dt}{\, dt}

\newcommand{\dvx}{\, d\vecx}

\newcommand{\sdW}{\,{\circ\hspace{-0.5pt}dW}}

\newcommand{\nn}{\nonumber}

\numberwithin{equation}{section}
\newtheorem{theorem}{Theorem}[section]
\newtheorem{lemma}[theorem]{Lemma}

\newtheorem{proposition}[theorem]{Proposition}
\newtheorem{remark}[theorem]{Remark}
\newtheorem{definition}[theorem]{Definition}

\title[FEM for stochastic Landau--Lifshitz--Gilbert
equation]{A FINITE ELEMENT APPROXIMATION FOR
THE STOCHASTIC LANDAU--LIFSHITZ--GILBERT EQUATION}

\author{Beniamin Goldys}
\address{School of Mathematics and Statistics,
         The University of Sydney,
         Sydney 2006, Australia}
\email{beniamin.goldys@sydney.edu.au}
\author{Kim-Ngan Le}
\address{School of Mathematics and Statistics,
         The University of New South Wales,
         Sydney 2052, Australia}
\email{n.le-kim@student.unsw.edu.au}
\author{Thanh Tran}
\address{School of Mathematics and Statistics,
         The University of New South Wales,
         Sydney 2052, Australia}
\email{thanh.tran@unsw.edu.au}

\date{\today}
\newtheorem{algorithm}{Algorithm}[section]

\begin{document}
\begin{abstract}
The stochastic Landau--Lifshitz--Gilbert (LLG) equation describes
the behaviour of the magnetisation under the influence of the
effective field containing of random fluctuations. 
We first transform the stochastic LLG equation into a partial differential equation with random coefficients (without the It\^o term). The resulting equation has time-differentiable solutions. 
We then propose a convergent
$\theta$-linear scheme for the numerical solution of the
reformulated equation. As a consequence,
we show the existence of weak
martingale solutions to the stochastic LLG equation.
A salient feature of this scheme is that it does not
involve solving a system of nonlinear algebraic equations, and
that no condition on time and space steps is required
when $\theta\in(\frac{1}{2},1]$.
Numerical results are presented to show the applicability of the method.

\end{abstract}
\maketitle
\section{Introduction}
The study of the theory of ferromagnetism involves the
study of the Landau--Lifshitz--Gilbert (LLG) equation
\cite{Gil55,LL35}. Let $D$ be a bounded domain in
$\R^d$, $d=2,3$, with a smooth boundary $\pa D$, and
let $\vecM : [0,T] \times D \goto \R^3$
denote the magnetisation of a ferromagnetic material
occupying the domain~$D$. The LLG equation takes the form
\begin{equation}\label{equ:llg}
\begin{cases}
&\vecM_t
=
\lambda_1\vecM\times\vecH_{\text{eff}}
-
\lambda_2\vecM\times(\vecM\times\vecH_{\text{eff}})
\quad\text{ in } D_T,\\
&\frac{\pa\vecM}{\pa \vecn}
= 0 \quad\text{ on } (0,T) \times \pa D,\\
&\vecM(0,\cdot)
=
\vecM_0 \quad\text{ in } D, 
\end{cases}
\end{equation}
where $\lambda_1\not=0$, $\lambda_2 >0$,
are constants, and $D_T = (0,T) \times D$. 
Here $\vecn$ denotes the outward normal vector on $\partial D$
and $\vecH_{\text{eff}}$ is the effective field; see
e.g. \cite{Cimrak_survey}.
Noting from~\eqref{equ:llg} that
$\snorm{\vecM(t,\vecx)}{} = $ const,
we assume that at time $t=0$ the material is saturated,
i.e.,
\begin{equation}\label{equ:m0}
\snorm{\vecM_0(\vecx)}{} = 1,
\quad \vecx\in D,
\end{equation}
and that
\begin{equation}\label{equ:M 1}
\snorm{\vecM(t,\vecx)}{} = 1,
\quad t\in(0,T), \ \vecx\in D.
\end{equation}

In the simplest situation when the energy functional
consists of the exchange energy only, the effective field
$\vecH_{\text{eff}}$ is in fact $\Delta\vecM$.
We recall that the stationary solutions
of~\eqref{equ:llg} (with $\vecH_{\text{eff}}=\Delta\vecM$) are in general not unique;
see~\cite{AloSoy92}. In the
theory of ferrormagnetism, it is important to describe
phase transitions between different equilibrium states
induced by thermal fluctuations of the effective field
$\vecH_{\text{eff}}$. It is therefore necessary to
modify $\vecH_{\text{eff}}$ to incorporate these random
fluctuations. In this paper, we
follow~\cite{BanBrzPro09,BrzGolJer12} to add a
noise to $\vecH_{\text{eff}}=\Delta\vecM$ so that
the stochastic version of the LLG equation takes the form
(see~\cite{BrzGolJer12})
\begin{equation}\label{E:1.1}
\begin{cases}
&d\vecM
=
\big(\lambda_1 \vecM\times \Delta \vecM
-
\lambda_2 \vecM\times(\vecM\times \Delta \vecM)\big)dt
+
(\vecM\times \vecg)\sdW(t),\\
&\frac{\pa\vecM}{\pa \vecn}
= 0 \quad\text{ on } (0,T) \times \pa D,\\
&\vecM(0,\cdot)
=
\vecM_0 \quad\text{ in } D, 
\end{cases}
\end{equation}
where $\vecg : D\goto\R^3$ is a given bounded function, and $W$ is a one-dimensional Wiener process.
Here $\sdW(t)$ stands for the Stratonovich
differential. In view of the property~\eqref{equ:M 1}
for the deterministic case, we require that $\vecM$ also
satisfies~\eqref{equ:M 1}.

We note that the driving noise can be
multi-dimensional; for simplicity of presentation, we
assume that it is one-dimensional.
This allows us to assume without loss of generality that
(see~\cite{BrzGolJer12})
\begin{equation}\label{equ:g 1}
\snorm{\vecg(\vecx)}{}
= 1, \quad \vecx\in D.
\end{equation}

In~\cite{GuoPu09}, an equation similar to~\eqref{E:1.1} is studied in the whole space $\R^d$, for any $d>0$. However, the noise considered in~\cite{GuoPu09} corresponds to a choice of $\vecg$ in~\eqref{E:1.1} to be constant across the domain $D$. It is also not clear how a solution to the stochastic LLG equation is defined.

In \cite{BrzGolJer12}, the authors use  the Faedo--Galerkin
approximations and the method of compactness to show that equation~\eqref{E:1.1} has a weak martingale solution. A convergent finite
element scheme for this problem is studied
in~\cite{BanBrzNekPro13,BanBrzPro09}. It is noted that this is a
{\em non-linear} scheme, namely the resulting system of algebraic equations is nonlinear, which requires the use of Newton's method.  


In this paper, we employ the finite element scheme
introduced in~\cite{AloJai06} (and further developed
in~\cite{Alo08}) for the deterministic LLG equation. 
Since this scheme
seeks to approximate the time derivative of the
magnetisation~$\vecM$, which is not well-defined in the stochastic
case, we first transform the stochastic LLG equation into a partial differential equation with random coefficients  (without the It\^o term). The resulting equation has time-differentiable solutions. Thus the $\theta$-linear scheme mentioned
above can be applied. As a consequence, we show 
the convergence of the finite element solutions to a weak martiangale solution to~\eqref{E:1.1}. In particular, we provide an independent proof of the existence of weak martingale solutions given in~\cite{BrzGolJer12}. It seems that our argument is simpler than theirs. 
Moreover, when $\theta>1/2$
no condition on
$h$ and $k$ is required for convergence of the method. We note that this scheme is also successfully applied to the
Maxwell--LLG equations in~\cite{LMDT13, LeTra12}.

 Another {\em linear} finite element scheme to solve the stochastic LLG equation is  announced by Alouges, De Bouard and Hocquet in~\cite{Alo14} after this paper was submitted. They propose a convergent time semi-discrete scheme for the stochastic LLG equation.
 In difference to our approach, they work directly with the It\^o equation~\eqref{E:1.1}.

The paper is organised as follows. In Section~\ref{sec:wea
sol} we define weak martingale solutions to~\eqref{E:1.1}
and state our main result. Section~\ref{sec:pre} prepares
sufficient tools which allow us to reformulate
equation~\eqref{E:1.1} to an equation with unknown
differentiable with respect to $t$.
Details of this reformulation are presented in
Section~\ref{sec:equ eqn}. We also show in this section how
a weak solution to~\eqref{E:1.1} can be obtained from a
weak solution of the reformulated form.
Section~\ref{sec:fin ele} introduces our finite element
scheme and presents a proof for the convergence of finite
element solutions to a weak solution of the reformulated
equation.
Section~\ref{sec:pro} is devoted to the proof of
the main theorem. Our numerical experiments are
presented in Section~\ref{sec:num}.

Throughout this paper, $c$ denotes a generic constant which
may take different values at different occurrences.

\section{Definition of a weak solution and the main
result}\label{sec:wea sol}
In this section we state the definition of a weak
solution to~\eqref{E:1.1} and our main result.
Before doing so, we introduce some Sobolev spaces and some notations.

For any $U\subset\R^d$, $d=1,2,3,\ldots$,
the function space $\mH^1(U)$ is defined as follows:
\begin{align*}
\mH^1(U)
&=
\left\{ \vecu\in\mL^2(U) : \frac{\p\vecu}{\p
x_i}\in \mL^2(U)\quad\text{for } i=1,2,3.
\right\}.
\end{align*}
Here, $\mL^2(U)$ is the usual space of Lebesgue squared
integrable functions defined on $U$ and taking values in $\R^3$.

\begin{remark}
For $\vecu, \vecv, \vecw\in \mH^1(D)$
we denote
\begin{align*}
\vecu\times\nabla\vecv
&:=
\left(
\vecu\times\frac{\partial\vecv}{\partial x_1},
\vecu\times\frac{\partial\vecv}{\partial x_2},
\vecu\times\frac{\partial\vecv}{\partial x_3}
\right) \\
\nabla\vecu\times\nabla\vecv
&:=
\sum_{i=1}^3
\frac{\partial\vecu}{\partial x_i}
\times
\frac{\partial\vecv}{\partial x_i} \\
\inpro{\vecu\times\nabla\vecv}{\nabla\vecw}_{\mL^2(D)}
&:=
\sum_{i=1}^3
\inpro{\vecu\times\frac{\partial\vecv}{\partial x_i}}%
{\frac{\partial\vecw}{\partial x_i}}_{\mL^2(D)}.
\end{align*}
\end{remark}
\begin{definition}\label{def:wea sol}
Given $T\in(0,\infty)$, a weak martingale solution
$(\Omega,\cF,(\cF_t)_{t\in[0,T]},\mP,W,\vecM)$
to~\eqref{E:1.1}, for the time interval $[0,T]$,
consists of
\begin{enumerate}
\renewcommand{\labelenumi}{(\alph{enumi})}
\item
a filtered probability space
$(\Omega,\cF,(\cF_t)_{t\in[0,T]},\mP)$ with the
filtration satisfying the usual conditions,
\item
a one-dimensional $(\cF_t)$-adapted Wiener process
$W=(W_t)_{t\in[0,T]}$,
\item
a progressively measurable
process $\vecM : [0,T]\times\Omega \goto \mL^2(D)$
\end{enumerate}
such that
\begin{enumerate}
\item
\quad
$\vecM(\cdot,\omega) \in C([0,T];\mH^{-1}(D))$
for $\mP$-a.s. $\omega\in\Omega$;
\item
\quad
$\mE\left(
\esssup_{t\in[0,T]}\|\nabla\vecM(t)\|^2_{\mL^2(D)}
\right) < \infty$;
\item
\quad
$|\vecM(t,x)| = 1$
for each $t \in [0,T]$, a.e. $x\in D$, and $\mP$-a.s.;
\item
\quad
for every  $t\in[0,T]$,
for all $\vecpsi\in\C_0^{\infty}(D)$,
$\mP$-a.s.:
\begin{align}\label{wE:1.1}
\inpro{\vecM(t)}{\vecpsi}_{\mL^2(D)}
-
\inpro{\vecM_0}{\vecpsi}_{\mL^2(D)}
&=
-\lambda_1
\int_0^t
\inpro{\vecM\times\nabla\vecM}{\nabla\vecpsi}_{\mL^2(D)}\ds\nn\\
&\quad-
\lambda_2
\int_0^t
\inpro{\vecM\times\nabla\vecM}{\nabla(\vecM\times\vecpsi)}_{\mL^2(D)}\ds\nn\\
&\quad+
\int_0^t
\inpro{\vecM\times\vecg}{\vecpsi}_{\mL^2(D)}\sdW(s).
\end{align}
\end{enumerate}
\end{definition}

As the main result of this paper, we will establish a finite element scheme defined a sequence of functions which are piecewise linear in both the space and time variables. We also prove that this sequence contains a subsequence converging to a weak martingale solution in the sense of Definition~\ref{def:wea sol}. A precise statement will be given in Theorem~\ref{the:mai}.

%

\section{Technical results}\label{sec:pre}
In this section we introduce and prove a few properties
of a transformation which will
be used in the next section to define a new variable
form $\vecM$.
\begin{lemma}\label{lem:G1}
Assume that $\vecg\in\mL^{\infty}(D)$.
Let
$G : \mL^2(D) \gotoo \mL^2(D)$ be defined by
\begin{equation}\label{equ:G}
G\vecu = \vecu\times\vecg \quad\forall\vecu\in\mL^2(D).
\end{equation}
Then the operator $G$ is well defined and for any
$\vecu, \vecv\in\mL^2(D)$,
\begin{align}
G^*
&=
-G\label{equ:adjoint}\\
\vecu\times G\vecv
&=
(\vecu\cdot\vecg)\vecv - (\vecu\cdot\vecv)\vecg,
\label{equ:G1}\\
\vecu\times G^2\vecv
&=
(\vecv\cdot\vecg)G\vecv - G\vecu\times G\vecv,
\label{equ:G3}\\
G\vecu\times G\vecv
&=
\big(\vecg\cdot(\vecu\times\vecv)\big)\vecg
=
G^2\vecu\times G^2\vecv,
\label{equ:G5} \\
G\vecu\times G^2\vecv
&=
\big((\vecg\cdot\vecu)(\vecg\cdot\vecv)
- (\vecu\cdot\vecv)\big)\vecg
=
- G^2\vecu\times G\vecv,
\label{equ:G4} \\
(G\vecu)\cdot\vecv
&=
-\vecu\cdot(G\vecv), \label{equ:G8a} \\
G^{2n+1}\vecu
&=
(-1)^n G\vecu, \quad n\ge0, \label{equ:G6} \\
G^{2n+2}\vecu
&=
(-1)^n G^2\vecu, \quad n\ge0. \label{equ:G7}
\end{align}
\end{lemma}
\begin{proof}
The proof can be done by using assumption~\eqref{equ:g 1} and
the following elementary identities: for all
$\veca, \vecb, \vecc\in\R^3$,
\begin{equation}\label{equ:abc}
\veca\times(\vecb\times\vecc) =
(\veca\cdot\vecc)\vecb
-
(\veca\cdot\vecb)\vecc
\end{equation}
and
\begin{equation}\label{equ:ele ide}
(\veca\times\vecb)\cdot\vecc
=
(\vecb\times\vecc)\cdot\veca
=
(\vecc\times\veca)\cdot\vecb.
\end{equation}
The last two properties \eqref{equ:G6} and
\eqref{equ:G7} also require the use of induction.
\end{proof}

For any $s\in\R$ the operator $e^{sG} : \mL^2(D) \goto \mL^2(D)$
has the following properties which can be proved by
using Lemma~\ref{lem:G1}.
\begin{lemma}\label{lem:G2}
For any $s\in\R$ and $\vecu,\vecv\in\mL^2(D)$,
\begin{align}
e^{sG}\vecu
&=
\vecu + (\sin s) G\vecu + (1-\cos s) G^2\vecu
\label{equ:G8} \\
e^{-sG} e^{sG} (\vecu)
&= \vecu\label{equ:G12}\\
\left(e^{sG}\right)^*
&=
e^{-sG}\label{equ:eadj}\\
e^{sG}G\vecu &= Ge^{sG}\vecu \label{equ:G9} \\
e^{sG}G^2\vecu &= G^2e^{sG}\vecu \label{equ:G10} \\
e^{sG}(\vecu\times\vecv)
&=
e^{sG}\vecu \times e^{sG}\vecv. \label{equ:G11}
\end{align}
\end{lemma}
\begin{proof}
By using Lemma~\ref{lem:G1} and Taylor's expansion we
obtain
\begin{align*}
e^{sG}\vecu
&=
\sum_{n=0}^\infty \frac{s^n}{n!} G^n\vecu \\
&=
\vecu
+
\sum_{k=0}^\infty \frac{s^{2k+1}}{(2k+1)!} G^{2k+1}\vecu
+
\sum_{k=0}^\infty \frac{s^{2k+2}}{(2k+2)!} G^{2k+2}\vecu
\\
&=
\vecu
+
\sum_{k=0}^\infty \frac{s^{2k+1}}{(2k+1)!} (-1)^k G\vecu
+
\sum_{k=0}^\infty \frac{s^{2k+2}}{(2k+2)!} (-1)^k
G^2\vecu \\
&=
\vecu + (\sin s) G\vecu + (1-\cos s) G^2\vecu,
\end{align*}
proving~\eqref{equ:G8}.
Equations~\eqref{equ:G12} and \eqref{equ:eadj} can be obtained by using~\eqref{equ:G8} and \eqref{equ:G7}.
Equations~\eqref{equ:G9} and \eqref{equ:G10} can be
obtained by using~\eqref{equ:G8} and the
definition~\eqref{equ:G}.

Finally, in order to prove~\eqref{equ:G11} we
use~\eqref{equ:G8} and~\eqref{equ:G3} to have
\begin{align*}
e^{sG}\vecu \times e^{sG}\vecv
&=
\vecu\times\vecv
+ \sin s\big(\vecu\times G\vecv + G\vecu\times\vecv\big)
+ (1-\cos s)
\big(\vecu\times G^2\vecv + G^2\vecu\times\vecv\big) \\
&\mbox{}\quad
+ \sin s(1-\cos s)
\big(G\vecu\times G^2\vecv + G^2\vecu\times G\vecv\big) \\
&\mbox{}\quad
+ \sin^2s \ G\vecu\times G\vecv
+ (1-\cos s)^2 G^2\vecu\times G^2\vecv \\
&=:
\vecu\times\vecv
+ T_1 + \cdots + T_5.
\end{align*}
Identities~\eqref{equ:G1} and~\eqref{equ:abc} give
$
T_1 = (\sin s) G(\vecu\times\vecv).
$
Identity~\eqref{equ:G4} gives $T_3=0$.
Using successivly~\eqref{equ:G4}, \eqref{equ:G3}
and~\eqref{equ:abc} we obtain
\[
T_2 + T_4 + T_5 = (1-\cos s) G^2(\vecu\times\vecv).
\]
Therefore,
\begin{align*}
e^{sG}\vecu \times e^{sG}\vecv
&=
\vecu\times\vecv
+ (\sin s) G(\vecu\times\vecv)
+ (1-\cos s) G^2(\vecu\times\vecv).
\end{align*}
Using~\eqref{equ:G8} we complete the proof of the lemma.
\end{proof}

In the proof of existence of weak solutions we also need the
following results (in the ``weak sense'') of the
operators $G$ and  $e^{sG}$.

\begin{lemma}\label{lem:wtd C}
Assume that $\vecg\in\mH^2(D)$.
For any $\vecu\in \mH^1(D)$ and $\vecv\in
\mW_0^{1,\infty}(D)$,
\begin{equation}\label{equ:C nab G}
\inpro{\nabla G\vecu}{\nabla\vecv}_{\mL^2(D)}
+
\inpro{\nabla\vecu}{\nabla G\vecv}_{\mL^2(D)}
=
-\inpro{C\vecu}{\vecv}_{\mL^2(D)}
\end{equation}
and
\begin{equation}\label{equ:CG GC}
\inpro{\nabla\vecu}{\nabla G^2\vecv}_{\mL^2(D)}
-
\inpro{\nabla G^2\vecu}{\nabla\vecv}_{\mL^2(D)}
=
\inpro{GC\vecu}{\vecv}_{\mL^2(D)}
+
\inpro{CG\vecu}{\vecv}_{\mL^2(D)},
\end{equation}
where
\[
C\vecu
=
\vecu\times\Delta\vecg
+ 2\sum_{i=1}^d
\frac{\pa\vecu}{\pa x_i}
\times
\frac{\pa\vecg}{\pa x_i}.
\]
\end{lemma}
\begin{proof}
Recalling the definition of $G$
(see~\eqref{equ:G}) and~\eqref{equ:ele ide}
we obtain
\begin{align*}
\inpro{\nabla G\vecu}{\nabla\vecv}_{\mL^2(D)}
+
\inpro{\nabla\vecu}{\nabla G\vecv}_{\mL^2(D)}
&=
\inpro{\nabla\vecu\times\vecg}{\nabla\vecv}_{\mL^2(D)}
+
\inpro{\vecu\times\nabla\vecg}{\nabla\vecv}_{\mL^2(D)}
\\
&\quad
+
\inpro{\nabla\vecu}{\nabla\vecv\times\vecg}_{\mL^2(D)}
+
\inpro{\nabla\vecu}{\vecv\times\nabla\vecg}_{\mL^2(D)}
\\
&=
\inpro{\vecu\times\nabla\vecg}{\nabla\vecv}_{\mL^2(D)}
-
\inpro{\nabla\vecu\times\nabla\vecg}{\vecv}_{\mL^2(D)}.
\end{align*}
By using Green's identity (noting that
$\vecv$ has zero trace on the boundary of $D$) and
the definition of $C$
we deduce
\begin{align*}
\inpro{\nabla G\vecu}{\nabla\vecv}_{\mL^2(D)}
+
\inpro{\nabla\vecu}{\nabla G\vecv}_{\mL^2(D)}
&=
-\inpro{\nabla(\vecu\times\nabla\vecg)}{\vecv}_{\mL^2(D)}
-
\inpro{\nabla\vecu\times\nabla\vecg}{\vecv}_{\mL^2(D)}\\
&=
-\inpro{\vecu\times\Delta\vecg)}{\vecv}_{\mL^2(D)}
-
2\inpro{\nabla\vecu\times\nabla\vecg}{\vecv}_{\mL^2(D)}\\
&=
-\inpro{C\vecu}{\vecv}_{\mL^2(D)},
\end{align*}
proving~\eqref{equ:C nab G}.

The proof of \eqref{equ:CG GC} is similarly. Firstly we
have from the definition of~$G$
\begin{align*}
\inpro{\nabla\vecu}{\nabla G^2\vecv}_{\mL^2(D)}
&-
\inpro{\nabla G^2\vecu}{\nabla\vecv}_{\mL^2(D)} \\
&=
\inpro{\nabla\vecu}{\nabla((\vecv\times\vecg)\times\vecg)}_{\mL^2(D)}
-
\inpro{\nabla
((\vecu\times\vecg)\times\vecg)}{\nabla\vecv}_{\mL^2(D)}\\
&=
\inpro{\nabla\vecu}{(\nabla\vecv\times\vecg)\times\vecg}_{\mL^2(D)}
+
\inpro{\nabla\vecu}{(\vecv\times\nabla\vecg)\times\vecg}_{\mL^2(D)}\\
&\quad+
\inpro{\nabla\vecu}{(\vecv\times\vecg)\times\nabla\vecg}_{\mL^2(D)}
-
\inpro{(\nabla\vecu\times\vecg)\times\vecg}{\nabla\vecv}_{\mL^2(D)}\\
&\quad-
\inpro{(\vecu\times\nabla\vecg)\times\vecg}{\nabla\vecv}_{\mL^2(D)}
-
\inpro{(\vecu\times\vecg)\times\nabla\vecg}{\nabla\vecv}_{\mL^2(D)}.
\end{align*}
Using again \eqref{equ:ele ide} and Green's identity
we deduce
\begin{align*}
\inpro{\nabla\vecu}{\nabla G^2\vecv}_{\mL^2(D)}
&-
\inpro{\nabla G^2\vecu}{\nabla\vecv}_{\mL^2(D)} \\
&=
\inpro{(\nabla\vecu\times\vecg)\times\vecg}{\nabla\vecv}_{\mL^2(D)}
+
\inpro{(\nabla\vecu\times\vecg)\times\nabla\vecg}{\vecv}_{\mL^2(D)}\\
&\quad+
\inpro{(\nabla\vecu\times\nabla\vecg)\times\vecg}{\vecv}_{\mL^2(D)}
-
\inpro{(\nabla\vecu\times\vecg)\times\vecg}{\nabla\vecv}_{\mL^2(D)}\\
&\quad+
\inpro{\nabla((\vecu\times\nabla\vecg)\times\vecg)}{\vecv}_{\mL^2(D)}
+
\inpro{\nabla((\vecu\times\vecg)\times\nabla\vecg)}{\vecv}_{\mL^2(D)}.
\end{align*}
Simple calculation reveals
\begin{align*}
\inpro{\nabla\vecu}{\nabla G^2\vecv}_{\mL^2(D)}
&-
\inpro{\nabla G^2\vecu}{\nabla\vecv}_{\mL^2(D)} \\
&=
2\inpro{(\nabla\vecu\times\vecg)\times\nabla\vecg}{\vecv}_{\mL^2(D)}
+
2\inpro{(\nabla\vecu\times\nabla\vecg)\times\vecg}{\vecv}_{\mL^2(D)}\\
&\quad+
\inpro{(\vecu\times\Delta\vecg)\times\vecg}{\vecv}_{\mL^2(D)}
+
2\inpro{(\vecu\times\nabla\vecg)\times\nabla\vecg}{\vecv}_{\mL^2(D)}\\
&\quad+
\inpro{(\vecu\times\vecg)\times\Delta\vecg}{\vecv}_{\mL^2(D)}\\
&=
\inpro{(\vecu\times\Delta\vecg)\times\vecg}{\vecv}_{\mL^2(D)}
+
2\inpro{(\nabla\vecu\times\nabla\vecg)\times\vecg}{\vecv}_{\mL^2(D)}\\
&\quad+
\inpro{(\vecu\times\vecg)\times\Delta\vecg}{\vecv}_{\mL^2(D)}
+
2\inpro{\nabla(\vecu\times\vecg)\times\nabla\vecg}{\vecv}_{\mL^2(D)}\\
&=
\inpro{GC\vecu}{\vecv}_{\mL^2(D)}
+
\inpro{CG\vecu}{\vecv}_{\mL^2(D)},
\end{align*}
proving the lemma.
\end{proof}


\begin{lemma}\label{lem:eC wea}
Assume that $\vecg\in\mH^2(D)$. For any $s\in\R$,
$\vecu\in\mH^1(D)$ and $\vecv\in\mW_0^{1,\infty}(D)$,
\begin{equation*}
\inpro{\wtd C(s,e^{-sG}\vecu)}{\vecv}_{\mL^2(D)}
=
\inpro{\nabla
e^{-sG}\vecu}{\nabla\vecv}_{\mL^2(D)}
-
\inpro{\nabla \vecu}{\nabla
e^{sG}\vecv}_{\mL^2(D)},
\end{equation*}
where
 \[
 \wtd C(s,\vecv)
 =
 e^{-sG}
 \big(
 (\sin s) C + (1-\cos s)(GC+CG)
 \big)\vecv.
 \]
 Here $C$ is defined in Lemma~\ref{lem:wtd C}.
\end{lemma}
\begin{proof}
Letting $\wtd\vecu=e^{-sG}\vecu$ and using the
definition of $\wtd C$ we have
\begin{align*}
\inpro{\wtd C(s,e^{-sG}\vecu)}{\vecv}_{\mL^2(D)}
&=
\inpro{\wtd C(s,\wtd\vecu)}{\vecv}_{\mL^2(D)}\\
&=
\inpro{e^{-sG}
\big((\sin s) C + (1-\cos s)(GC+CG) \big)
\wtd\vecu}{\vecv}_{\mL^2(D)}.
\end{align*}
Using successively~\eqref{equ:eadj} and Lemma~\ref{lem:wtd C}
 we deduce
\begin{align*}
\inpro{\wtd C(s,e^{-sG}\vecu)}{\vecv}_{\mL^2(D)}
&=
\sin s\inpro{ C\wtd\vecu}{e^{sG}\vecv}_{\mL^2(D)}
+
(1-\cos s)\inpro{(GC+ CG)\wtd\vecu}{e^{sG}\vecv}_{\mL^2(D)}\\
&=
-\sin s
\left[
\inpro{\nabla G\wtd\vecu}{\nabla
e^{sG}\vecv}_{\mL^2(D)}
+
\inpro{\nabla\wtd\vecu}{\nabla
Ge^{sG}\vecv}_{\mL^2(D)}
\right]\\
&\quad+(1-\cos s)
\left[
\inpro{\nabla\wtd\vecu}{\nabla
G^2e^{sG}\vecv}_{\mL^2(D)}
-
\inpro{\nabla G^2\wtd\vecu}{\nabla
e^{sG}\vecv}_{\mL^2(D)}
\right].
\end{align*}
Simple calculation yields
\begin{align*}
\inpro{\wtd C(s,e^{-sG}\vecu)}{\vecv}_{\mL^2(D)}
&=
\inpro{\nabla\wtd\vecu}%
{\nabla ((I-\sin sG+(1-\cos s)G^2))e^{sG}\vecv}_{\mL^2(D)} \\
&\quad -
\inpro{\nabla(I+(\sin s)G+(1-\cos s)G^2)\wtd\vecu}%
{\nabla e^{sG}\vecv}_{\mL^2(D)}.
\end{align*}
Using \eqref{equ:G8} and~\eqref{equ:G12} we obtain
\begin{align*}
\inpro{\wtd C(s,e^{-sG}\vecu)}{\vecv}_{\mL^2(D)}
&=
\inpro{\nabla\wtd\vecu}{\nabla \vecv}_{\mL^2(D)}
-
\inpro{\nabla e^{sG}\wtd\vecu}{\nabla
e^{sG}\vecv}_{\mL^2(D)}.
\end{align*}
The desired result now follows from the definition of
$\wtd\vecu$.
\end{proof}
\section{Equivalence of weak solutions}\label{sec:equ eqn}
In this section we use the operator $G$ defined in the
preceding section to define a new process $\vecm$ from
$\vecM$.
Let
\begin{equation}\label{equ:vecm}
\vecm(t,\vecx) = e^{-W(t)G}\vecM(t,\vecx) \quad\forall t
\ge0, \ a.e. \vecx\in D.
\end{equation}
It turns out that with this new variable, the
differential $dW(t)$ vanishes in the
partial differential equation satisfied by $\vecm$.
Moreover, it will be seen that $\vecm$ is differentiable with
respect to $t$.
In the next lemma, we introduce the equation satisfied by $\vecm$ so that $\vecM$ is a solution to~\eqref{E:1.1} in the sense of~\eqref{wE:1.1}.
\begin{lemma}\label{lem:4.2}
Assume that $\vecg\in\mW^{2,\infty}(D)$.
If  
 $\vecm(\cdot,\omega)\in H^1(0,T;\mL^2(D))\cap L^2(0,T;\mH^1(D))$, 
for $\mP$-a.s. $\omega\in\Omega$, satisfies
\begin{align}\label{InE:14}
\inpro{\vecm_t}{\vecpsi}_{\mL^2(D_T)}
&+
\lambda_1
\inpro{\vecm\times\nabla\vecm}{\nabla\vecpsi}_{\mL^2(D_T)}
+
\lambda_2
\inpro{\vecm\times\nabla\vecm}{\nabla(\vecm\times\vecpsi)}_{\mL^2(D_T)}\nn\\
&-
\inpro{F(t,\vecm)}{\vecpsi}_{\mL^2(D_T)}
=0 \quad \forall \vecpsi\in L^2(0,T;\mW^{1,\infty}(D)),
\quad \mP\text{-a.s.},
\end{align}
where
\begin{equation}\label{equ:R}
 F(t,\vecm)
 =
 \lambda_1\vecm\times \wtd C(W(t),\vecm(t,\cdot))
 -
 \lambda_2\vecm\times(\vecm\times \wtd C(W(t),\vecm(t,\cdot))).
 \end{equation}
Then $\vecM = e^{W(t)G} \vecm$ satisfies \eqref{wE:1.1} $\mP$-a.s..
\end{lemma}
\begin{proof}
Since $e^{W(t)G}$ is a semimartingale and $\vecm$ is absolutely continuous, 
using It\^o's formula for $\vecM=e^{W(t)G}\vecm$ (see e.g.~\cite{DaPrato92}),  we deduce
\begin{align*}
\vecM(t)
=
\vecM(0)
+
\int_0^t Ge^{W(s)G}\vecm \ dW(s)
+ \int_0^t\big(e^{W(s)G} \vecm_t 
+\frac12 G^2 e^{W(s)G}\vecm\big) \ds,
\end{align*}
where the first integral is the It\^o integral and the second is the Bochner
integral.  
We recall
the relation between the Stratonovich and
It\^o differentials
 \begin{equation}\label{equ:Str Ito}
 (G\vecu)\sdW(s)
 =
 \frac12 G'(\vecu)[G\vecu] \ds
 +
 G(\vecu) \ dW(s)
 \end{equation}
 where
 \[
 G'(\vecu)[G\vecu] = G^2\vecu
 \]
 to write the above equation in the Stratonovich form as
\[
\vecM(t)
=
\vecM(0)
+
\int_0^t G\vecM \sdW(s) 
+
\int_0^t e^{W(s)G} \vecm_t\ds.
\]
Multiplying both sides by a test function
$\vecpsi\in\C_0^{\infty}(D)$ and integrating over $D$ we
obtain
\begin{align}\label{equ:wIto}
\inpro{\vecM(t)}{\vecpsi}_{\mL^2(D)}
&=
\inpro{\vecM(0)}{\vecpsi}_{\mL^2(D)}
+
\int_0^t
\inpro{G\vecM}{\vecpsi}_{\mL^2(D)}\sdW(s)\nn\\
&\quad+
\int_0^t
\inpro{e^{W(s)G}\vecm_t}{\vecpsi}_{\mL^2(D)}\ds\nn\\
&=
\inpro{\vecM(0)}{\vecpsi}_{\mL^2(D)}
+
\int_0^t
\inpro{G\vecM}{\vecpsi}_{\mL^2(D)}\sdW(s)\nn\\
&\quad+
\int_0^t
\inpro{\vecm_t}{e^{-W(s)G}\vecpsi}_{\mL^2(D)}\ds
\end{align}
where in the last step we used~\eqref{equ:eadj}
and~\eqref{equ:G8a}.
On the other hand, it follows from~\eqref{InE:14} that,
for all $\vecxi\in  L^2(0,t;\mW^{1,\infty}(D))$,
\begin{align}\label{equ:dm xi}
\int_0^t
\inpro{\vecm_t}{\vecxi}_{\mL^2(D)}\ds
&=
-\lambda_1
\int_0^t
\inpro{\vecm\times\nabla\vecm}{\nabla\vecxi}_{\mL^2(D)}\ds\nn\\
&\quad-\lambda_2
\int_0^t
\inpro{\vecm\times\nabla\vecm}{\nabla(\vecm\times\vecxi)}_{\mL^2(D)}\ds\nn\\
&\quad+
\int_0^t
\inpro{F(s,\vecm)}{\vecxi}_{\mL^2(D)}\ds.
\end{align}
Using~\eqref{equ:dm xi} with $\vecxi=e^{-W(s)G}\vecpsi$ for the last term on the right hand side of~\eqref{equ:wIto} we deduce  
\begin{align*}
\inpro{\vecM(t)}{\vecpsi}_{\mL^2(D)}
&=
\inpro{\vecM(0)}{\vecpsi}_{\mL^2(D)}
+
\int_0^t
\inpro{G\vecM}{\vecpsi}_{\mL^2(D)}\sdW(s)\nn\\
&\quad-\lambda_1
\int_0^t
\inpro{\vecm\times\nabla\vecm}{\nabla\big(e^{-W(s)G}\vecpsi\big)}_{\mL^2(D)}\ds \\
&\quad
-\lambda_2
\int_0^t
\inpro{\vecm\times\nabla\vecm}{\nabla(\vecm\times e^{-W(s)G}\vecpsi)}_{\mL^2(D)}\ds\nn\\
&\quad+
\int_0^t
\inpro{F(s,\vecm)}{e^{-W(s)G}\vecpsi}_{\mL^2(D)}\ds.
\end{align*}
It follows from the definition~\eqref{equ:R} that
\begin{align}\label{equ:dM t1 t2}
\inpro{\vecM(t)}{\vecpsi}_{\mL^2(D)}
&=
\inpro{\vecM(0)}{\vecpsi}_{\mL^2(D)}
+
\int_0^t
\inpro{G\vecM}{\vecpsi}_{\mL^2(D)}\sdW(s)\nn\\
&\quad+
 \int_0^t
 \big(
  \lambda_1(T_1+T_2) + \lambda_2 (T_3+T_4) 
  \big)\ds,
\end{align}
where
\begin{align*}
T_1&=
-
\inpro{\vecm\times\nabla\vecm}{\nabla
\big(e^{-W(s)G}\vecpsi\big)}_{\mL^2(D)} \\
T_2&=
\inpro{\vecm\times
\wtd C(W(s),\vecm)}{e^{-W(s)G}\vecpsi}_{\mL^2(D)}\nn\\
T_3&=
-
\inpro{\vecm\times\nabla\vecm}{\nabla(\vecm\times
e^{-W(s)G}\vecpsi)}_{\mL^2(D)} \\
T_4&=
-
\inpro{\vecm\times(\vecm\times
\wtd C(W(s),\vecm))}{e^{-W(s)G}\vecpsi}_{\mL^2(D)},
\end{align*}
with $\wtd C$ defined in Lemma~\ref{lem:eC wea}.
By using~\eqref{equ:ele ide}, the definition
$\vecm(s,\cdot)=e^{-W(s)G}\vecM(s,\cdot)$, and~\eqref{equ:G11} we obtain
\begin{align*}
T_2
&=
\inpro{\wtd C(W(s),\vecm)}{e^{-W(s)G}\vecpsi\times\vecm}_{\mL^2(D)}\\
&=
\inpro{\wtd C(W(s),e^{-W(s)G}\vecM)}%
{e^{-W(s)G}\big(\vecpsi\times \vecM\big)}_{\mL^2(D)}.
\end{align*}
Lemma~\ref{lem:eC wea} then gives
\begin{align*}
T_2
&=
\inpro{\nabla e^{-W(s)G}\vecM}{\nabla
e^{-W(s)G}(\vecpsi\times \vecM)}_{\mL^2(D)}
-
\inpro{\nabla\vecM}{\nabla(\vecpsi\times\vecM)}_{\mL^2(D)}
\\
&=
-T_1
-
\inpro{\nabla\vecM}{\nabla(\vecpsi\times\vecM)}_{
D},
\end{align*}
implying
\[
T_1 + T_2
=
- \inpro{\nabla\vecM}{\nabla(\vecpsi\times\vecM)}_{\mL^2(D)}
=
- \inpro{\vecM\times\nabla\vecM}{\nabla\vecpsi}_{\mL^2(D)},
\]
where we used~\eqref{equ:ele ide}.
Similarly we have
\[
T_3 + T_4
=
-
\inpro{\vecM\times\nabla\vecM}{\nabla(\vecM\times\vecpsi)}_{\mL^2(D)}.
\]
Equation~\eqref{equ:dM t1 t2} then yields
\begin{align*}
\inpro{\vecM(t)}{\vecpsi}_{\mL^2(D)}
-
\inpro{\vecM_0}{\vecpsi}_{\mL^2(D)}
&=
-\lambda_1
\int_0^t
\inpro{\vecM\times\nabla\vecM}{\nabla\vecpsi}_{\mL^2(D)}\ds\nn\\
&\quad-
\lambda_2
\int_0^t
\inpro{\vecM\times\nabla\vecM}{\nabla(\vecM\times\vecpsi)}_{\mL^2(D)}\ds\nn\\
&\quad+
\int_0^t
\inpro{\vecM\times\vecg}{\vecpsi}_{\mL^2(D)}\sdW(s),
\end{align*}
which complete the proof.
\end{proof}
The following result can be easily proved.
\begin{lemma}\label{lem:m 1}
Under the assumption~\eqref{equ:g 1}, 
 $\vecM$ satisfies
\[
\snorm{\vecM(t,\vecx)}{} = 1
\quad\forall t\ge0, \ a.e.\, \vecx\in D,\, \mP-\text{a.s.}
\]
if and only if $\vecm$ defined in~\eqref{equ:vecm} satisfies
\[
\snorm{\vecm(t,\vecx)}{} = 1
\quad\forall t\ge0, \ a.e. \vecx\in D,\, \mP-\text{a.s.}.
\]
\end{lemma}
\begin{proof}
 The proof can be done by using~\eqref{equ:G12} and~\eqref{equ:eadj}. 
\end{proof}
In the next lemma we provide a relationship between equation~\eqref{InE:14} and its Gilbert form.
\begin{lemma}\label{lem:4.1}
Let $\vecm\in\mH^1(D_T)$ satisfy
\begin{equation}\label{equ:m 1}
|\vecm(t,\vecx)| = 1, \quad t\in(0,T), \ \vecx\in D,
\end{equation}
and
\begin{align}\label{InE:13}
\lambda_1\inpro{\vecm_t}{\vecvarphi}_{\mL^2(D_T)}
&+
\lambda_2\inpro{\vecm\times\vecm_t}{\vecvarphi}_{\mL^2(D_T)} \nn\\
&=
\mu \inpro{\nabla\vecm}{\nabla(\vecm\times\vecvarphi)}_{\mL^2(D_T)}
+
\lambda_1\inpro{F(t,\vecm)}{\vecvarphi}_{\mL^2(D_T)}
\nn\\
&+
\lambda_2
\inpro{\vecm\times F(t,\vecm)}{\vecvarphi}_{\mL^2(D_T)} \quad\forall \vecvarphi\in L^2(0,T;\mH^1(D)),
\end{align}
where $\mu=\lambda_1^2+\lambda_2^2$. Then $\vecm$ satisfies~\eqref{InE:14}.
\end{lemma}
\begin{proof}
For each $\vecpsi\in L^2(0,T;\mW^{1,\infty}(D))$, using Lemma~\ref{lem:4.0}
in the Appendix, there exists
$\vecvarphi\in L^2(0,T;\mH^1(D))$ such that
\begin{equation}\label{Equ:phi}
\lambda_1{\vecvarphi}
+
\lambda_2{\vecvarphi}\times\vecm
=\vecpsi.
\end{equation}
By using \eqref{equ:ele ide}
we can write~\eqref{InE:13} as
\begin{align}\label{InE:15}
&\quad\inpro{\vecm_t}
{\lambda_1{\vecvarphi}+\lambda_2{\vecvarphi}\times\vecm}_{\mL^2(D_T)}
+
\lambda_1
\inpro{\vecm\times\nabla\vecm}
{\nabla(\lambda_1{\vecvarphi})}_{\mL^2(D_T)}\nn\\
&+\lambda_2
\inpro{\nabla\vecm}
{\nabla(\lambda_2{\vecvarphi}\times\vecm)}_{\mL^2(D_T)}
-
\inpro{F(t,\vecm)}
{\lambda_1{\vecvarphi}+\lambda_2{\vecvarphi}\times\vecm}_{\mL^2(D_T)}
=0.
\end{align}
On the other hand, by using~\eqref{equ:abc}
and~\eqref{equ:m 1} we can show that
\begin{align}\label{InE:16}
\lambda_1
\inpro{\vecm\times\nabla\vecm}
{\nabla(\lambda_2{\vecvarphi}\times\vecm)}_{\mL^2(D_T)}
&+\lambda_2
\inpro{\nabla\vecm}
{\nabla(\lambda_1{\vecvarphi})}_{\mL^2(D_T)}\nn\\
&-\lambda_2
\inpro{|\nabla\vecm|^2\vecm}
{\lambda_1{\vecvarphi}}_{\mL^2(D_T)}=0.
\end{align}
Moreover, we have
\begin{equation}\label{InE:17}
-\lambda_2
\inpro{|\nabla\vecm|^2\vecm}
{\lambda_2{\vecvarphi}\times\vecm}_{\mL^2(D_T)}=0.
\end{equation}
Summing \eqref{InE:15}--\eqref{InE:17} gives
\begin{align*}
&\quad\inpro{\vecm_t}
{\lambda_1{\vecvarphi}+\lambda_2{\vecvarphi}\times\vecm}_{\mL^2(D_T)}
+
\lambda_1
\inpro{\vecm\times\nabla\vecm}
{\nabla(\lambda_1{\vecvarphi}+\lambda_2{\vecvarphi}\times\vecm)}_{\mL^2(D_T)}\nn\\
&+
\lambda_2
\inpro{\nabla\vecm}
{\nabla(\lambda_1{\vecvarphi}+\lambda_2{\vecvarphi}\times\vecm)}_{\mL^2(D_T)}
-
\lambda_2
\inpro{|\nabla\vecm|^2\vecm}
{\lambda_1{\vecvarphi}+\lambda_2{\vecvarphi}\times\vecm}_{\mL^2(D_T)}\nn\\
&-
\inpro{F(t,\vecm)}
{\lambda_1{\vecvarphi}+\lambda_2{\vecvarphi}\times\vecm}_{\mL^2(D_T)}
=0
\end{align*}
The desired equation \eqref{InE:14} follows by noting
\eqref{Equ:phi} and using~\eqref{equ:abc}, \eqref{equ:ele ide},
and~\eqref{equ:m 1}.
\end{proof}

\begin{remark}\label{rem:LLL}
By using~\eqref{equ:ele ide} we can rewrite~\eqref{InE:13}
as
\begin{align}\label{E:1.3a}
\lambda_1\inpro{\vecm\times\vecm_t}{\vecw}_{\mL^2(D_T)}
&-
\lambda_2\inpro{\vecm_t}{\vecw}_{\mL^2(D_T)} \nn\\
&=
\mu \inpro{\nabla\vecm}{\nabla\vecw}_{\mL^2(D_T)}
+
\inpro{R(t,\vecm)}{\vecw}_{\mL^2(D_T)},
\end{align}
where
\begin{equation*}
 R(t,\vecm)
 =
 \lambda_2^2\vecm\times(\vecm\times \wtd C(W(t),\vecm(t,\cdot)))
 -
 \lambda_1^2\wtd C(W(t),\vecm(t,\cdot)),
\end{equation*}
and $\vecw=\vecm\times\vecvarphi$ for
$\vecvarphi\in L^2(0,T;\mH^1(D))$.
It is noted that
$\vecw\cdot\vecm=0$. This property will be exploited later
in the design of the finite element scheme.
\end{remark}
A martingale solution to~\eqref{InE:13} is defined as follows.
\begin{definition}\label{def:wea solm}
Given $T^*\in(0,\infty)$, a martingale solution
to~\eqref{InE:13} for the time interval $[0,T^*]$,
denoted by
$(\Omega,\cF,(\cF_t)_{t\in[0,T^*]},\mP,W,\vecm)$,
consists of
\begin{enumerate}
\renewcommand{\labelenumi}{(\alph{enumi})}
\item
a filtered probability space
$(\Omega,\cF,(\cF_t)_{t\in[0,T^*]},\mP)$ with the
filtration satisfying the usual conditions,
\item
a one-dimensional $(\cF_t)$-adapted Wiener process
$W=(W_t)_{t\in[0,T^*]}$,
\item
a progressively measurable
process $\vecm : [0,T^*]\times\Omega \goto \mL^2(D)$
\end{enumerate}
such that
\begin{enumerate}
\item
\quad
$\vecm(\cdot,\omega) \in H^1(0,T^*;\mL^2(D))\cap L^2(0,T^*;\mH^1(D))$
for $\mP$-a.s. $\omega\in\Omega$;
\item
\quad
$\mE\left(
\esssup_{t\in[0,T^*]}\|\nabla\vecm(t)\|^2_{\mL^2(D)}
\right) < \infty$;
\item
\quad
$|\vecm(t,x)| = 1$
for each $t \in [0,T^*]$, a.e. $x\in D$, and $\mP$-a.s.;
\item
\quad
$\vecm(0,\cdot)=\vecM_0$ in $D$
\item
\quad
for every  $T\in[0,T^*]$, $\vecm$ satisfies~\eqref{InE:13}
\end{enumerate}
\end{definition}
We state the following lemma as a consequence of Lemmas~\ref{lem:4.1},~\ref{lem:m 1} and~\ref{lem:4.2}.
\begin{lemma}\label{lem:equi}
If $\vecm$ is a solution of~\eqref{InE:13} in the sense of Definition~\ref{def:wea solm}, then $\vecM = e^{W(t)G} \vecm$ is a weak martingale solution of~\eqref{E:1.1} in the sense of Definition~\ref{def:wea sol}.
\end{lemma}
\begin{proof}
Since $\vecm$ satisfies $(1)$ in Definition~\ref{def:wea solm}, there holds
\begin{equation}\label{equ:min}
\mP\big(\vecm\in L^2(0,T^*;\mH^1(D))\cap H^1(0,T^*;\mL^2(D))\big)=1.
\end{equation}
It follows from~\eqref{equ:vecm},~\eqref{equ:min} and Lemma~\ref{lem:imbed}, that
\[
\mP\big(\vecM\in C\big([0,T^*];\mH^{-1}(D)\big)\big)=1.
\]
By using Lemmas~\ref{lem:4.2},~\ref{lem:m 1} and~\ref{lem:4.1}, we deduce that $\vecM$  satisfies $(2)$, $(3)$, $(4)$ in Definition~\ref{def:wea sol}, which completes the proof.
\end{proof}
Thanks to the above lemma, we now solve equation~\eqref{InE:13} instead of ~\eqref{wE:1.1}.

\section{The finite element scheme}\label{sec:fin ele}
In this section we design a finite element scheme to find
approximate solutions to~\eqref{InE:13}. More precisely, we
prove in the next section that the
finite element solutions converge to a solution
of~\eqref{InE:13}. Then thanks to Lemma~\ref{lem:equi}
we obtain a weak solution of~\eqref{wE:1.1}.

Let $\mT_h$ be a regular tetrahedrization of the domain
$D$ into tetrahedra of maximal mesh-size $h$.
We denote by $\cN_h := \{\vecx_1,\ldots,\vecx_N\}$ the set
of vertices and by $\cM_h :=\{ e_1, \ldots , e_M \}$ the
set of edges.

Before introducing the finite element scheme,
 we state the following result proved by
Bartels~\cite{Bart05} which will be used in the analysis.
\begin{lemma}\label{lem:bar}
Assume that
\begin{equation}\label{E:CondTe}
\int_{D} \nabla\phi_i\cdot\nabla\phi_j\dvx \leq 0
\quad\text{for all}\quad i,j \in \{1,2,\cdots,J\}\text{ and
} i\not= j .
\end{equation}
Then for all $\vecu\in\mV_h$ satisfying
$|\vecu(\vecx_l)|\geq 1$, $ l=1,2,\cdots,J$, there holds
\begin{equation}\label{E:InE}
\int_{D}\left|\nabla
I_{\mV_h}\left(\frac{\vecu}{|\vecu|}\right)\right|^2\dvx
\leq
\int_{D}|\nabla\vecu|^2\dvx.
\end{equation}
\end{lemma}
\noindent
When $d=2$, condition~\eqref{E:CondTe} holds for Delaunay
triangulation. 
(Roughly speaking, a Delaunay triangulation is a triangulation
in which no vertex is contained inside the circumference of any triangle.) 
When $d=3$, condition~\eqref{E:CondTe} holds if all dihedral angles of
the tetrahedra in $\mT_h|_{\mL^2(D)}$ are less than or equal
to $\pi/2$; see~\cite{Bart05}.
In the sequel we assume that~\eqref{E:CondTe} holds.

To discretize the equation~\eqref{InE:13}, we
introduce the finite element space
$\mV_h\subset\mH^1(D)$ which is the space of all
continuous piecewise linear functions on $\mT_h$. A
basis for $\mV_h$ can be chosen to be $(\phi_n)_{1\leq
n\leq N}$, where
 $\phi_n(\vecx_m)=\delta_{n,m}.$  Here $\delta_{n,m}$
stands for
the Kronecker symbol.
The interpolation operator from
$\C^0(D)$ onto  $\mV_h$, denoted by $I_{\mV_h}$, is defined by
\[
I_{\mV_h}(\vecv)=\sum_{n=1}^N \vecv(\vecx_n)\phi_n(\vecx)
\quad\forall \vecv\in \mathbb C^0(D,\mathbb R^3) .
\]

Fixing a positive integer $J$, we choose the time step
$k$ to be $k=T/J$ and define $t_j=jk$, $j=0,\cdots,J$. For
$j=1,2,\ldots,J$, the solution $\vecm (t_j,\cdot)$
is approximated by $\vecm^{(j)}_h\in\mV_h$, which is
computed as follows.

Since
\[
\vecm_t(t_j,\cdot)
\approx
\frac{\vecm(t_{j+1},\cdot)-\vecm(t_j,\cdot)}{k}
\approx
\frac{\vecm_h^{(j+1)}-\vecm_h^{(j)}}{k},
\]
we can define $\vecm_h^{(j+1)}$ from $\vecm_h^{(j)}$ by
\begin{equation}\label{equ:mjp1}
\vecm_h^{(j+1)}
=
\vecm_h^{(j)} + k \vecv_h^{(j)},
\end{equation}
where $\vecv_h^{(j)}$ is an approximation of
$\vecm_t(t_j,\cdot)$. Hence it suffices to propose a scheme
to compute $\vecv_h^{(j)}$.

Motivated by the property $\vecm_t\cdot\vecm=0$, 
we
will find $\vecv_h^{(j)}$ in the space $\mW^{(j)}_h$ defined
by
\begin{equation}\label{equ:Whj}
 \mW_h^{(j)}
 :=
 \left\{\vecw\in \mV_h \mid
 \vecw(\vecx_n)\cdot\vecm_h^{(j)}(\vecx_n)=0,
 \ n = 1,\ldots,N \right\}.
\end{equation}
Given $\vecm_h^{(j)}\in\mV_h$,
we use~\eqref{E:1.3a} to define $\vecv_h^{(j)}$ instead
of using~\eqref{InE:13} so that the same test and trial
functions can be used (see Remark~\ref{rem:LLL}).
Hence we define by $\vecv_h^{(j)}\in\mW_h^{(j)}$
\begin{align}\label{E:1.5}
-\lambda_1
\inpro{\vecm_h^{(j)}\times\vecv_h^{(j)}}{\vecw_h^{(j)}}_{\mL^2(D)}
&+
\lambda_2
\inpro{\vecv_h^{(j)}}{\vecw_h^{(j)}}_{\mL^2(D)}
=
-\mu
\inpro{\nabla (\vecm_h^{(j)}+k\theta \vecv_h^{(j)})}
{\nabla\vecw_h^{(j)}}_{\mL^2(D)} \nn\\
&-
\inpro{R_{h,k}(t_j,\vecm_h^{(j)})}{\vecw_h^{(j)}}_{\mL^2
(D)},
\end{align}
where the approximation $R_{h,k}(t_j,\vecm_h^{(j)})$ to
$R(t_j,\vecm(t_j,\cdot))$ needs to be defined.

Considering the piecewise constant approximation $W_k(t)$
of $W(t)$, namely,
\begin{equation}\label{Def:W}
W_k(t)=W(t_j),\quad t\in[t_j,t_{j+1}),
\end{equation}
we define, for each $\vecu\in\mV_h$,
\begin{align*}
G_h\vecu &= \vecu\times I_{\mV_h}(\vecg) \\
C_h(\vecu)
&=
\vecu\times I_{\mV_h}(\Delta\vecg)
+
2\nabla\vecu\times I_{\mV_h}(\nabla\vecg).
\end{align*}
We can then define $R_{h,k}$ by
\begin{equation}\label{equ:Fk}
R_{h,k}(t,\vecu)
=
\lambda_2^2 \vecu\times(\vecu\times
\wtd C_{h,k}(t,\vecu))
-
\lambda_1^2 \wtd C_{h,k}(t,\vecu),
\end{equation}
where
\begin{align}
D_{h,k}(t,\vecu)
&=
\left(\sin W_k(t) C_h + (1-\cos W_k(t))(G_hC_h+C_hG_h) \right)\vecu
\label{equ:Dhk} \\
\wtd C_{h,k}(t,\vecu)
&=
\left(I-\sin W_k(t) G_h + (1-\cos
W_k(t))G_h^2\right)D_{h,k}(t,\vecu).
\label{equ:Chk}
\end{align}

We summarise the algorithm as follows.

\bigskip
\begin{algorithm}\label{Algo:1}
\mbox{}
\begin{description}
\item[Step 1]
Set $j=0$.
Choose $\vecm^{(0)}_h=I_{\mV_h}\vecm_0$.
\item[Step 2]
Find $\vecv_h^{(j)}\in \mW_h^{(j)}$
satisfying~\eqref{E:1.5}.\label{A:2}
\item[Step 3] \label{A:4}
Define
\begin{equation*}
\vecm_h^{(j+1)}(\vecx)
:=
\sum_{n=1}^N
\frac{\vecm_h^{(j)}(\vecx_n)+k\vecv_h^{(j)}(\vecx_n)}
{\left|\vecm_h^{(j)}(\vecx_n)+k\vecv_h^{(j)}(\vecx_n)\right|}
\phi_n(\vecx).
\end{equation*}
\item[Step 4]
Set $j=j+1$, and return to Step 2 if $j<J$. Stop if
$j=J$.
\end{description}
\end{algorithm}
Since $\left|\vecm_h^{(0)}(x_n)\right|=1$ and
$\vecv_h^{(j)}(x_n)\cdot\vecm_h^{(j)}(x_n)=0$ for all
$n=1,\ldots,N$ and $j=0,\ldots,J$, we obtain (by induction)
\begin{equation}\label{equ:mhj 1}
\left |\vecm_h^{(j)}(x_n)+k\vecv_h^{(j)}(x_n)\right| \ge 1
\quad\text{and}\quad
\left |\vecm_h^{(j)}(x_n)\right |=1,
\quad j = 0,\ldots,J.
\end{equation}
In particular, the above inequality shows that
the algorithm is well defined.

We finish this section by proving the following three lemmas
concerning some properties of
$\vecm_h^{(j)}$ and $R_{h,k}$.
\begin{lemma}\label{lem:mhj}
For any $j=0,\ldots,J$,
\[
\norm{\vecm_h^{(j)}}{\mL^{\infty}(D)} \le 1
\quad\text{and}\quad
\norm{\vecm_h^{(j)}}{\mL^2(D)} \le |D|,
\]
where $|D|$ denotes the measure of $D$.
\end{lemma}
\begin{proof}
The first inequality follows from~\eqref{equ:mhj 1} and the
second can be obtained by integrating $\vecm_h^{(j)}(\vecx)$ over $D$.
\end{proof}
\begin{lemma}\label{lem:Fk}
Assume that $\vecg\in\mW^{2,\infty}(D)$.
There exists a deterministic constant $c$ depending
only on $\vecg$, such that for any $j=0,\cdots,J$,
\begin{equation}\label{equ:boundFk}
\left\|R_{h,k}(t_j,\vecm_h^{(j)})\right\|_{\mL^2(D)}^2
\leq
c +
c\left \| \nabla \vecm_h^{(j)} \right \| _{\mL^2(D)}^2,
\quad \mP-\text{a.s.}
\end{equation}
\end{lemma}
\begin{proof}
Recalling the definition~\eqref{equ:Fk} we have by
using the triangular inequality and Lemma~\ref{lem:mhj}
\begin{align}\label{equ:Fk2}
\left\|R_{h,k}(t_j,\vecm_h^{(j)})\right\|_{\mL^2(D)}^2
&\leq
2
\left\|
\lambda_2^2 \vecm_h^{(j)}\times(\vecm_h^{(j)}\times
\wtd C_{h,k}(t_j,\vecm_h^{(j)}))
\right\|_{\mL^2(D)}^2
+2
\left\|
\lambda_1^2 \wtd C_{h,k}(t_j,\vecm_h^{(j)})
\right\|_{\mL^2(D)}^2\nn\\
&\leq
2(\lambda_1^4+\lambda_2^4)
\left\|
\wtd C_{h,k}(t_j,\vecm_h^{(j)})
\right\|_{\mL^2(D)}^2.
\end{align}
We now estimate $\left\|\wtd C_{h,k}(t_j,\vecm_h^{(j)})\right\|_{\mL^2(D)}^2$.
From \eqref{equ:Chk} we have
\begin{align*}
\wtd C_{h,k}(t_j,\vecm_h^{(j)})
&=
D_{h,k}(t_j,\vecm_h^{(j)})
-
\sin W_k(t_j) D_{h,k}(t_j,\vecm_h^{(j)})\times \vecg_h\\
&\quad
+
(1-\cos W_k(t_j))(D_{h,k}(t_j,\vecm_h^{(j)})\times
\vecg_h)\times \vecg_h.
\end{align*}
The Cauchy--Schwarz inequality and Lemma~\ref{lem:Ih vh}
then yield
\begin{align}\label{equ:Chk2}
\left\|
\wtd C_{h,k}(t_j,\vecm_h^{(j)})
\right\|_{\mL^2(D)}^2
&\leq
\left( 1+\sin^2 W_k(t_j)+(1-\cos W_k(t_j))^2\right)
\left(
\left\|
D_{h,k}(t_j,\vecm_h^{(j)})
\right\|_{\mL^2(D)}^2 \right. \nn\\
&\quad
\left.
+\left\|
D_{h,k}(t_j,\vecm_h^{(j)})\times \vecg_h
\right\|_{\mL^2(D)}^2
+\left\|
(D_{h,k}(t_j,\vecm_h^{(j)})\times \vecg_h)\times \vecg_h
\right\|_{\mL^2(D)}^2
\right)\nn\\
&\leq
c\left (1+\left\|\vecg\right\|_{\mL^{\infty}(D)}^2+
\left\|\vecg\right\|_{\mL^{\infty}(D)}^4\right)
\left\|
D_{h,k}(t_j,\vecm_h^{(j)})
\right\|_{\mL^2(D)}^2.
\end{align}
By using the same technique we can prove
\begin{align}\label{equ:Dhk2}
\left\| D_{h,k}(t_j,\vecm_h^{(j)})\right\|_{\mL^2(D)}^2
&\leq
c\ \left(
\left\|\Delta\vecg\right\|_{\mL^{\infty}(D)}^2
+
\left\|\Delta\vecg\right\|_{\mL^{\infty}(D)}^2
\left\|\vecg\right\|_{\mL^{\infty}(D)}^2
\right)
\nn\\
&\quad
+
c \ \left(
\left\|\nabla\vecg\right\|_{\mL^{\infty}(D)}^2
+
\left\|\nabla\vecg\right\|_{\mL^{\infty}(D)}^2
\left\|\vecg\right\|_{\mL^{\infty}(D)}^2
\right)
\left\|\nabla\vecm_h^{(j)}\right\|_{\mL^2(D)}^2.
\end{align}
From~\eqref{equ:Fk2}, \eqref{equ:Chk2}, and
\eqref{equ:Dhk2}, we deduce
the desired result.
\end{proof}
\begin{lemma}\label{lem:3.2}
There exist a deterministic constant $c$
depending on $\vecm_0$, $\vecg$, $\mu_1$, $\mu_2$ and $T$
such that for $j=1,\ldots,J$,
\begin{align*}
\left \| \nabla \vecm_h^{(j)} \right \| _{\mL^2(D)}^2
+
\sum_{i=0}^{j-1} k \left \| v_h^{(i)}\right\|_{\mL^2(D)}^2
+
k ^2 (2\theta-1)
\sum_{i=0}^{j-1}\left\| \nabla \vecv_h^{(i)} \right\|_{\mL^2(D)}^2
\leq
c, \quad\mP-\text{a.s.}
\end{align*}
\end{lemma}
\begin{proof}
Taking $\vecw_h^{(j)}=\vecv_h^{(j)}$
in equation \eqref{E:1.5} yields
to the following identity
\begin{align*}
\lambda_2\left \| \vecv_h^{(j)}\right\|_{\mL^2(D)}^2
&=
-\mu \inpro{\nabla \vecm_h^{(j)}}{\nabla\vecv_h^{(j)}}_{\mL^2(D)}
-\mu \theta k \norm{\nabla\vecv_h^{(j)}}{\mL^2(D)}^2 
-
\inpro{R_{h,k}(t_j,\vecm_h^{(j)})}{\vecv_h^{(j)}}_{\mL^2(D)},
\end{align*}
or equivalently
\begin{align}\label{E:3.1}
\inpro{\nabla \vecm_h^{(j)}}{\nabla \vecv_h^{(j)}}_{\mL^2(D)}
&=
-\lambda_2\mu^{-1} \left \| \vecv_h^{(j)}\right\|_{\mL^2(D)}^2
- \theta k \left \| \nabla \vecv_h^{(j)} \right
\|_{\mL^2(D)}^2 
-
\mu^{-1}
\inpro{R_{h,k}(t_j,\vecm_h^{(j)})}{\vecv_h^{(j)}}_{\mL^2(D)}.
\end{align}
From Lemma~\ref{lem:bar} it follows that
\[
\left\|\nabla\vecm_h^{(j+1)}\right\|_{\mL^2(D)}^2
\leq
\left\|\nabla(\vecm_h^{(j)}+k\vecv_h^{(j)})\right\|_{\mL^2(D)}^2,
\]
and therefore, by using~\eqref{E:3.1}, we deduce
\begin{align*}
\left \| \nabla \vecm_h^{(j+1)} \right \| _{\mL^2(D)}^2
& \leq \left \| \nabla \vecm_h^{(j)} \right \|_{\mL^2(D)}^2
+ k ^2 \left \| \nabla \vecv_h^{(j)} \right \| _{\mL^2(D)}^2
+ 2k \inpro{\nabla\vecm_h^{(j)}}{\nabla \vecv_h^{(j)}}_{\mL^2(D)} \\
&\leq
\left \| \nabla \vecm_h^{(j)} \right \|_{\mL^2(D)}^2
+ k ^2 \left \| \nabla \vecv_h^{(j)} \right \| _{\mL^2(D)}^2
- 2\lambda_2\mu^{-1} k \left \| \vecv_h^{(j)}\right\|_{\mL^2(D)}^2 \\
&\quad
-2 \theta k^2  \left \| \nabla \vecv_h^{(j)} \right \|_{\mL^2(D)}^2
-
2k\mu^{-1}
\inpro{R_{h,k}(t_j,\vecm_h^{(j)})}{\vecv_h^{(j)}}_{\mL^2(D)}.\\
&\leq
\left \| \nabla \vecm_h^{(j)} \right \|_{\mL^2(D)}^2
+ k ^2 (1-2\theta)
\left \| \nabla \vecv_h^{(j)} \right \| _{\mL^2(D)}^2
- 2\lambda_2\mu^{-1} k \left \| \vecv_h^{(j)}\right\|_{\mL^2(D)}^2 \\
&\quad
-
2k\mu^{-1}
\inpro{R_{h,k}(t_j,\vecm_h^{(j)})}{\vecv_h^{(j)}}_{\mL^2(D)}.
\end{align*}
By using the elementary inequality $2ab\le \al^{-1}a^2+\al
b^2$ (for any $\al>0$) to the
last term on the right hand side, we deduce
\begin{align*}
\left \| \nabla \vecm_h^{(j+1)} \right \| _{\mL^2(D)}^2
\leq
&\left \| \nabla \vecm_h^{(j)} \right \|_{\mL^2(D)}^2
+ k ^2 (1-2\theta)\left \| \nabla \vecv_h^{(j)} \right \| _{\mL^2(D)}^2
- 2\lambda_2\mu^{-1} k \left \| \vecv_h^{(j)}\right\|_{\mL^2(D)}^2\\
&+
\mu^{-1} k
\left(
\lambda_2^{-1}
\left\|R_{h,k}(t_j,\vecm_h^{(j)})\right \|_{\mL^2(D)}^2
+
\lambda_2
\left \| \vecv_h^{(j)}\right\|_{\mL^2(D)}^2 \right),
\end{align*}
which implies
\begin{align*}
\left \| \nabla \vecm_h^{(j+1)} \right \| _{\mL^2(D)}^2
&+
k ^2 (2\theta-1)\left \| \nabla \vecv_h^{(j)} \right \| _{\mL^2(D)}^2
+\lambda_2\mu^{-1} k \left \| \vecv_h^{(j)}\right\|_{\mL
^2(D)}^2 \\
&\leq
\left \| \nabla \vecm_h^{(j)} \right \|_{\mL^2(D)}^2
+
k\mu^{-1}\lambda_2^{-1}
\left\|R_{h,k}(t_j,\vecm_h^{(j)})\right \|_{\mL^2(D)}^2.
\end{align*}
Replacing $j$ by $i$ in the above inequality and summing
for $i$ from $0$ to $j-1$ yields
\begin{align*}
\left \| \nabla \vecm_h^{(j)} \right \| _{\mL^2(D)}^2
&+
\sum_{i=0}^{j-1} k \left \| v_h^{(i)}\right\|_{\mL^2(D)}^2
+
k ^2 (2\theta-1)\sum_{i=0}^{j-1}
\left \| \nabla \vecv_h^{(i)} \right \| _{\mL^2(D)}^2 \nn\\
&\leq
c\left \| \nabla \vecm_h^{(0)} \right \| _{\mL^2(D)}^2
+
ck\sum_{i=0}^{j-1}
\left\|R_{h,k}(t_i,\vecm_h^i)\right \|_{\mL^2(D)}^2 .
\end{align*}
Since $\vecm_0\in\mH^2(D)$ it can be shown that there
exists a deterministic constant $c$ depending only on $\vecm_0$
such that
\begin{equation}\label{equ:cm0}
\norm{\nabla\vecm_h^{(0)}}{\mL^2(D)}
\le c.
\end{equation}
By using~\eqref{equ:boundFk} we deduce
\begin{align}\label{InE:3.3}
\left \| \nabla \vecm_h^{(j)} \right \| _{\mL^2(D)}^2
&+
\sum_{i=0}^{j-1} k \left \| v_h^{(i)}\right\|_{\mL^2(D)}^2
+
k^2 (2\theta-1)\sum_{i=0}^{j-1}
\left \| \nabla \vecv_h^{(i)} \right \| _{\mL^2(D)}^2 \nn\\
&\leq
c + ck\sum_{i=0}^{j-1}
\left(1 +
\left\|\nabla\vecm_h^i\right \|_{\mL^2(D)}^2 \right) \nn\\
&\leq
c
+
ck
\sum_{i=0}^{j-1}
\left\|\nabla\vecm_h^i\right \|_{\mL^2(D)}^2.
\end{align}
By using induction and~\eqref{equ:cm0} we can show that
\[
\norm{\nabla\vecm_h^i}{\mL^2(D)}^2
\le
c(1+ck)^i.
\]
Summing over $i$ from 0 to $j-1$ and using $1+x\le e^x$ we
obtain
\[
k\sum_{i=0}^{j-1}
\left\|\nabla\vecm_h^i\right \|_{\mL^2(D)}^2
\le
ck \frac{(1+ck)^j-1}{ck}
\le
e^{ckJ} = c.
\]
This together with~\eqref{InE:3.3} gives the desired
result.
\end{proof}
\section{The main result}\label{sec:pro}
In this section, we will construct from the finite element function $\vecm_h^{(j)}$ a sequence of functions which converges (in some sense) to a weak martingale solution of~\eqref{E:1.1} in the sense of Definition~\ref{def:wea sol}.

The discrete solutions $\vecm_h^{(j)}$ and $\vecv_h^{(j)}$
constructed via Algorithm~\ref{Algo:1}
are interpolated in time in the following definition.
\begin{definition}\label{def:mhk}
For all $x\in D$ and all $t\in[0,T]$, let
$j\in \{ 0,...,J] \}$ be
such that  $t \in [t_j, t_{j+1})$. We then define
\begin{align*}
\vecm_{h,k}(t,x)
&:=
\frac{t-t_j}{k}\vecm_h^{(j+1)}(x)
+
\frac{t_{j+1}-t}{k}\vecm_h^{(j)}(x), \\
\vecm_{h,k}^{-}(t,x)
&:=
\vecm_h^{(j)}(x), \\
\vecv_{h,k}(t,x)
&:=
\vecv_h^{(j)}(x).
\end{align*}
\end{definition}
The above sequences have the following obvious bounds.
\begin{lemma}\label{lem:3.2a}
There exist a deterministic constant $c$
depending on $\vecm_0$, $\vecg$, $\mu_1$, $\mu_2$ and $T$
such that for all $\theta\in[0,1]$,
\begin{align*}
\norm{\vecm_{h,k}^*}{\mL^2(D_T)}^2 +
\left \| \nabla \vecm_{h,k}^* \right \| _{\mL^2(D_T)}^2
+
\left \| \vecv_{h,k}\right\|_{\mL^2(D_T)}^2
+
k (2\theta-1)
\left\| \nabla \vecv_{h,k} \right\|_{\mL^2(D_T)}^2
\leq
c,\quad \mP\text{-a.s.},
\end{align*}
where $\vecm_{h,k}^*=\vecm_{h,k}$ or $\vecm_{h,k}^-$.
In particular,
when $\theta\in[0,\frac{1}{2})$,
\begin{align*}
\norm{\vecm_{h,k}^*}{\mL^2(D_T)}^2
+
&\left \| \nabla \vecm_{h,k}^* \right \| _{\mL^2(D_T)}^2
+
\big(1+(2\theta-1)kh^{-2}\big)
\left \| \vecv_{h,k}\right\|_{\mL^2(D_T)}^2
\leq
c,\quad \mP\text{-a.s.}.
\end{align*}
\end{lemma}
\begin{proof}
It is easy to see that
\[
\norm{\vecm_{h,k}^-}{\mL^2(D_T)}^2
=
k \sum_{i=0}^{J-1} \norm{\vecm_{h}^{(i)}}{\mL^2(D)}^2
\quad\text{and}\quad
\norm{\vecv_{h,k}}{\mL^2(D_T)}^2
=
k \sum_{i=0}^{J-1} \norm{\vecv_{h}^{(i)}}{\mL^2(D)}^2.
\]
Both inequalities are direct consequences of
Definition~\ref{def:mhk},
Lemmas~\ref{lem:mhj}, and~\ref{lem:3.2}, noting that
the second inequality requires the use of
the inverse estimate (see e.g.~\cite{Johnson87})
\[
\norm{\nabla\vecv_{h}^{(i)}}{\mL^2(D)}^2
\leq
ch^{-2}
\norm{\vecv_{h}^{(i)}}{\mL^2(D)}^2.
\]

\end{proof}

The next lemma provides a bound of $\vecm_{h,k}$ in the
$\mH^1$-norm and relationships between $\vecm_{h,k}^-$,
$\vecm_{h,k}$ and $\vecv_{h,k}$.
\begin{lemma}\label{lem:3.4}
Assume that $h$ and $k$ go to $0$ with a further condition $k=o(h^2)$ when $\theta\in[0,\frac{1}{2})$ and no condition otherwise. The sequences $\{\vecm_{h,k}\}$, $\{\vecm_{h,k}^{-}\}$, and
$\{\vecv_{h,k}\}$ defined in
Definition~\ref{def:mhk} satisfy the following properties
$\mP$-a.s.
\begin{align}
\norm{\vecm_{h,k}}{\mH^1(D_T)}
&\le c, \label{equ:mhk h1} \\
\norm{\vecm_{h,k}-\vecm_{h,k}^-}{\mL^2(D_T)}
&\le ck, \label{equ:mhk mhkm} \\
\norm{\vecv_{h,k}-\pa_t\vecm_{h,k}}{\mL^1(D_T)}
&\le ck, \label{equ:vhk mhk} \\
\norm{|\vecm_{h,k}|-1}{\mL^2(D_T)}
&\le c(h+k). \label{equ:mhk 1}
\end{align}
\end{lemma}
\begin{proof}
Due to Lemma~\ref{lem:3.2a} to prove \eqref{equ:mhk h1}
it suffices to show the boundedness of
$\norm{\pa_t\vecm_{h,k}}{\mL^2(D_T)}$. First we note that, for
$t\in [t_j,t_{j+1})$,
\begin{equation*}
\norm{\partial_t\vecm_{h,k}(t)}{\mL^2(D)}
=
\bnorm{\frac{\vecm_h^{(j+1)}-\vecm_h^{(j)}}{k}}{\mL^2(D)}.
\end{equation*}
Furthermore, it can be shown that (see e.g.
\cite{LeTra12})
\[
\left|
\frac{\vecm_h^{(j+1)}(\vecx_n)-\vecm_h^{(j)}(\vecx_n)}{k}
\right|
\leq
\left|
\vecv_h^{(j)}(\vecx_n)
\right|
\quad
\forall n=1,2,\cdots, N, \quad j=0,\ldots,J.
\]
The above inequality together with Lemma~\ref{lem:nor equ} in
the Appendix yields
\[
\norm{\partial_t\vecm_{h,k}(t)}{\mL^2(D)}
\le
c \norm{\vecv_h^{(j)}}{\mL^2(D)}
=
c \norm{\vecv_{h,k}(t)}{\mL^2(D)}.
\]
The bound now follows from Lemma~\ref{lem:3.2a}.

Inequality~\eqref{equ:mhk mhkm} can be deduced
from~\eqref{equ:mhk h1} by noting that
for $t\in [t_j,t_{j+1})$,
\begin{align*}
\left|
\vecm_{h,k}(t,\vecx)-\vecm_{h,k}^-(t,\vecx)
\right|
&=
\left|
(t-t_j)
\frac{\vecm_h^{(j+1)}(\vecx)-\vecm_h^{(j)}(\vecx)}{k}
\right|
\leq
k\left| \partial_t\vecm_{h,k}(t,\vecx) \right|.
\end{align*}
Therefore, \eqref{equ:mhk mhkm} is a consequence
of~\eqref{equ:mhk h1}.

To prove~\eqref{equ:vhk mhk} we first note that the
definition of~$\vecm_h^{(j+1)}$ and~\eqref{equ:mhj 1} give
\begin{equation*}
\left|
\vecm_h^{(j+1)}(\vecx_n)
-
\vecm_h^{(j)}(\vecx_n)
-k\vecv_h^{(j)}(\vecx_n)
\right|
=
\left|
\vecm_h^{(j)}(\vecx_n)+k\vecv_h^{(j)}(\vecx_n)
\right|
-1.
\end{equation*}
On the other hand from the properties
$\snorm{\vecm_h^{(j)}(\vecx_n)}{}=1$, see~\eqref{equ:mhj
1}, and
$\vecm_h^{(j)}(\vecx_n)\cdot\vecv_h^{(j)}(\vecx_n)=0$,
see~\eqref{equ:Whj}, we deduce
\begin{align*}
\left|
\vecm_h^{(j)}(\vecx_n)+k\vecv_h^{(j)}(\vecx_n)
\right|
&=
\left(1+k^2\left|\vecv_h^{(j)}(\vecx_n)\right|^2\right)^{1/2}
\leq
1+\frac{1}{2}k^2\left|\vecv_h^{(j)}(\vecx_n)\right|^2.
\end{align*}
Therefore,
\begin{equation*}
\left|
\frac{\vecm_h^{(j+1)}(\vecx_n)-\vecm_h^{(j)}(\vecx_n)}{k}
-\vecv_h^{(j)}(\vecx_n)
\right|
\leq
\frac{1}{2}k\left|\vecv_h^{(j)}(\vecx_n)\right|^2.
\end{equation*}
Using Lemma~\ref{lem:nor equ} successively for $p=1$ and
$p=2$ we obtain, for $t\in [t_j,t_{j+1})$,
\begin{align*}
\left\|
\partial_t\vecm_{h,k}(t)
-
\vecv_{h,k}(t)
\right\|_{\mL^1(D)}
\leq
ck
\left\|
\vecv_{h,k}(t)
\right\|^2_{\mL^2(D)}.
\end{align*}
By integrating over $[t_j,t_{j+1})$, summing up over $j$,
and using Lemma~\ref{lem:3.2} we infer~\eqref{equ:vhk mhk}.

Finally, to prove~\eqref{equ:mhk 1} we note
that if $\vecx_n$ is a vertex of an element $K$ and
$t\in[t_j,t_{j+1})$ then
\begin{align*}
\Big|
|\vecm_{h,k}^-(t,\vecx)|-1
\Big|^2
&=
\Big|
|\vecm_{h,k}^-(t,\vecx)|-|\vecm_{h,k}^-(t,\vecx_n)|
\Big|^2 \\
&\leq
ch^2
\left|
\nabla\vecm_{h,k}^-(t,\vecx)
\right|^2
=
ch^2
\left|
\nabla\vecm_{h}^{(j)}(\vecx)
\right|^2
\quad\forall\vecx\in K.
\end{align*}
Integrating over $D_T$ and using Lemma~\ref{lem:3.2} we
obtain
\begin{equation*}
\bnorm{|\vecm_{h,k}^-|-1}{\mL^2(D_T)}
\leq
ch.
\end{equation*}
The required result~\eqref{equ:mhk 1} now follows
from~\eqref{equ:mhk mhkm} by using the triangle inequality.
\end{proof}
The following two Lemmas~\ref{lem:3.7} and ~\ref{lem:3.7a} show that ~$\vecm_{h,k}^-$ and~$\vecm_{h,k}$, respectively, satisfy a discrete form of~\eqref{InE:13}.
\begin{lemma}\label{lem:3.7}
Assume that $h$ and $k$ go to 0 with the following
conditions
\begin{equation}\label{equ:theta}
\begin{cases}
k = o(h^2) & \quad\text{when } 0 \le \theta < 1/2, \\
k = o(h) & \quad\text{when } \theta = 1/2, \\
\text{no condition} & \quad\text{when } 1/2<\theta\le1.
\end{cases}
\end{equation}
Then for any $\vecpsi \in C_0^\infty\big((0,T);\C^{\infty}(D)\big)$, 
\begin{align*}
-\lambda_1\inpro{\vecm_{h,k}^-\times\vecv_{h,k}}
{\vecm_{h,k}^-\times\vecpsi}_{\mL^2(D_T)}
&+
\lambda_2\inpro{\vecv_{h,k}}
{\vecm_{h,k}^-\times\vecpsi}_{\mL^2(D_T)}
\nonumber\\
&+
\mu\inpro{\nabla(\vecm_{h,k}^-+k\theta\vecv_{h,k})}
{\nabla(\vecm_{h,k}^-\times\vecpsi)}_{\mL^2(D_T)} \nn\\
&+
\inpro{R_{h,k}(.,\vecm_{h,k}^-)}
{\vecm_{h,k}^-\times\vecpsi}_{\mL^2(D_T)}
= O(h),\quad \mP\text{-a.s.}.
\end{align*}
\end{lemma}
\begin{proof}
For $t\in[t_j, t_{j+1})$, we use
equation~\eqref{E:1.5} with
$\vecw_h^{(j)}=I_{\mV_h}\big(\vecm_{h,k}^-(t,\cdot)\times\vecpsi(t,\cdot)\big)$
to have
\begin{align*}
-\lambda_1
&\inpro{\vecm_{h,k}^-(t,\cdot)\times\vecv_{h,k}(t,\cdot)}
{I_{\mV_h}\big(\vecm_{h,k}^-(t,\cdot)\times\vecpsi(t,\cdot)\big)}_{\mL^2(D)} \\
&+
\lambda_2\inpro{\vecv_{h,k}(t,\cdot)}
{I_{\mV_h}\big(\vecm_{h,k}^-(t,\cdot)\times\vecpsi(t,\cdot)\big)}_{\mL^2(D)}\nonumber\\
&+
\mu\inpro{\nabla(\vecm_{h,k}^-(t,\cdot)+k\theta\vecv_{h,k}(t,\cdot))}
{\nabla
I_{\mV_h}\big(\vecm_{h,k}^-(t,\cdot)\times\vecpsi(t,\cdot)\big)}_{\mL^2(D)}\\
&+
\inpro{R_{h,k}(t_j,\vecm_{h,k}^-(t,\cdot))}
{I_{\mV_h}\big(\vecm_{h,k}^-(t,\cdot)\times\vecpsi(t,\cdot)\big)}_{\mL^2(D)}
= 0.
\end{align*}
Integrating both sides of the above equation over
$(t_j,t_{j+1})$ and summing over $j=0,\ldots,J-1$ we deduce
\begin{align*}
-\lambda_1
&\inpro{\vecm_{h,k}^-\times\vecv_{h,k}}
{I_{\mV_h}\big(\vecm_{h,k}^-\times\vecpsi\big)}_{\mL^2(D_T)}
+
\lambda_2\inpro{\vecv_{h,k}}
{I_{\mV_h}\big(\vecm_{h,k}^-\times\vecpsi\big)}_{\mL^2(D_T)}\nonumber\\
&\quad \quad +
\mu\inpro{\nabla(\vecm_{h,k}^-+k\theta\vecv_{h,k})}
{\nabla I_{\mV_h}\big(\vecm_{h,k}^-\times\vecpsi\big)}_{\mL^2(D_T)}\nn\\
&\quad \quad +
\inpro{R_{h,k}(\cdot,\vecm_{h,k}^-)}
{I_{\mV_h}\big(\vecm_{h,k}^-\times\vecpsi\big)}_{\mL^2(D_T)}
= 0.
\end{align*}
This implies
\begin{align*}
-\lambda_1
&\inpro{\vecm_{h,k}^-\times\vecv_{h,k}}
{\vecm_{h,k}^-\times\vecpsi}_{\mL^2(D_T)}
+
\lambda_2\inpro{\vecv_{h,k}}
{\vecm_{h,k}^-\times\vecpsi}_{\mL^2(D_T)}\nonumber\\
&+
\mu\inpro{\nabla(\vecm_{h,k}^-+k\theta\vecv_{h,k})}
{\nabla(\vecm_{h,k}^-\times\vecpsi)}_{\mL^2(D_T)}\nn\\
&\quad \quad+
\inpro{R_{h,k}(.,\vecm_{h,k}^-)}
{\vecm_{h,k}^-\times\vecpsi}_{\mL^2(D_T)}
=I_1+I_2+I_3
\end{align*}
where
\begin{align*}
I_1
&=
\inpro{-\lambda_1\vecm_{h,k}^-\times\vecv_{h,k}
+
\lambda_2\vecv_{h,k}}
{\vecm_{h,k}^-\times\vecpsi
-
I_{\mV_h}(\vecm_{h,k}^-\times\vecpsi)}_{\mL^2(D_T)},\\
I_2
&=
\mu\inpro{\nabla(\vecm_{h,k}^-+k\theta\vecv_{h,k})}
{\nabla(\vecm_{h,k}^-\times\vecpsi-
I_{\mV_h}(\vecm_{h,k}^-\times\vecpsi))}_{\mL^2(D_T)},\\
I_3
&=
\inpro{R_{h,k}(.,\vecm_{h,k}^-)}
{\vecm_{h,k}^-\times\vecpsi
-
I_{\mV_h}(\vecm_{h,k}^-\times\vecpsi)}_{\mL^2(D_T)}.
\end{align*}
Hence it suffices to prove that $I_i=O(h)$ for $i=1,2,3$.
First, by using Lemma~\ref{lem:mhj} we obtain
\[
\norm{\vecm_{h,k}^-}{\mL^\infty(D_T)}
\le
\sup_{0\le j \le J} \norm{\vecm_h^{(j)}}{\mL^\infty(D)}
\leq
1.
\]
This inequality, Lemma~\ref{lem:3.2a} and Lemma~\ref{lem:Ih
vh} yield
\begin{align*}
|I_1|
&\le
c\left(\norm{\vecm_{h,k}^-}{\mL^\infty(D_T)}+1\right)
\norm{\vecv_{h,k}}{\mL^2(D_T)}
\norm{\vecm_{h,k}^-\times\vecpsi
-
I_{\mV_h}(\vecm_{h,k}^-\times\vecpsi)}{\mL^2(D_T)} \\
&\le
c\norm{\vecv_{h,k}}{\mL^2(D_T)}
\norm{\vecm_{h,k}^-\times\vecpsi
-
I_{\mV_h}(\vecm_{h,k}^-\times\vecpsi)}{\mL^2(D_T)} \\
&\le
ch.
\end{align*}
The bounds for $I_2$ and $I_3$ can be carried out similarly
by using Lemma~\ref{lem:3.2a} and Lemma~\ref{lem:Fk},
respectively, by noting that when $\theta\in[0,\frac{1}{2}]$, a bound of $k\left\| \nabla \vecv_{h,k} \right\|_{\mL^2(D_T)}$ can be deduced from the inverse estimate as follows
\[
k\left\| \nabla \vecv_{h,k} \right\|_{\mL^2(D_T)}
\leq
ckh^{-1}
\left \| \vecv_{h,k}\right\|_{\mL^2(D_T)}
\leq ckh^{-1}.
\]
This completes the proof of the lemma.
\end{proof}
\begin{lemma}\label{lem:3.7a}
Assume that $h$ and $k$ go to 0 satisfying~\eqref{equ:theta}.
Then for any $\vecpsi \in C_0^\infty\big((0,T);\C^{\infty}(D)\big)$,
\begin{align}\label{InE:10}
&-\lambda_1
\inpro{\vecm_{h,k}\times\pa_t\vecm_{h,k}}
{\vecm_{h,k}\times\vecpsi}_{\mL^2(D_T)}
+
\lambda_2\inpro{\pa_t\vecm_{h,k}}
{\vecm_{h,k}\times\vecpsi}_{\mL^2(D_T)}
\nonumber\\
&+
\mu\inpro{\nabla\vecm_{h,k}}
{\nabla(\vecm_{h,k}\times\vecpsi)}_{\mL^2(D_T)}
+
\inpro{R_{h,k}(\cdot,\vecm_{h,k})}
{\vecm_{h,k}\times\vecpsi}_{\mL^2(D_T)}
= O(hk),\quad \mP\text{-a.s.}.
\end{align}
\end{lemma}
\begin{proof}
From Lemma~\ref{lem:3.7} it follows that
\begin{align*}
&-\lambda_1\inpro{\vecm_{h,k}\times\pa_t\vecm_{h,k}}
{\vecm_{h,k}\times\vecpsi}_{\mL^2(D_T)}
+
\lambda_2\inpro{\pa_t\vecm_{h,k}}
{\vecm_{h,k}\times\vecpsi}_{\mL^2(D_T)}
\nonumber\\
&+
\mu\inpro{\nabla\vecm_{h,k}}
{\nabla(\vecm_{h,k}\times\vecpsi)}_{\mL^2(D_T)}
+
\inpro{R_{h,k}(\cdot,\vecm_{h,k})}
{\vecm_{h,k}\times\vecpsi}_{\mL^2(D_T)}
=
I_1+\cdots+I_4,
\end{align*}
where
\begin{align*}
I_1
&=
-\lambda_1\inpro{\vecm_{h,k}^-\times\vecv_{h,k}}
{\vecm_{h,k}^-\times\vecpsi}_{\mL^2(D_T)}
+
\lambda_1\inpro{\vecm_{h,k}\times\pa_t\vecm_{h,k}}
{\vecm_{h,k}\times\vecpsi}_{\mL^2(D_T)},\\
I_2
&=
\lambda_2\inpro{\vecv_{h,k}}
{\vecm_{h,k}^-\times\vecpsi}_{\mL^2(D_T)}
-
\lambda_2\inpro{\pa_t\vecm_{h,k}}
{\vecm_{h,k}\times\vecpsi}_{\mL^2(D_T)}
,\\
I_3
&=
\mu\inpro{\nabla(\vecm_{h,k}^-+k\theta\vecv_{h,k})}
{\nabla(\vecm_{h,k}^-\times\vecpsi)}_{\mL^2(D_T)}
-
\mu\inpro{\nabla\vecm_{h,k}}
{\nabla(\vecm_{h,k}\times\vecpsi)}_{\mL^2(D_T)}
,\\
I_4
&=
\inpro{R_{h,k}(.,\vecm_{h,k}^-)}
{\vecm_{h,k}^-\times\vecpsi}_{\mL^2(D_T)}
-
\inpro{R_{h,k}(\cdot,\vecm_{h,k})}
{\vecm_{h,k}\times\vecpsi}_{\mL^2(D_T)}
.
\end{align*}
Hence it suffices to prove that $I_i=O(h)$ for $i=1,\cdots,4$. Frist, by using triangle inequality we obtain
\begin{align*}
\lambda_1^{-1}|I_1|
&\leq
\left|
\inpro{(\vecm_{h,k}^--\vecm_{h,k})\times\vecv_{h,k}}
{\vecm_{h,k}^-\times\vecpsi}_{\mL^2(D_T)}
\right|\\
&\quad+
\left|
\inpro{\vecm_{h,k}\times\vecv_{h,k}}
{(\vecm_{h,k}^--\vecm_{h,k})\times\vecpsi}_{\mL^2(D_T)}
\right|\\
&\quad+
\left|
\inpro{\vecm_{h,k}\times(\vecv_{h,k}-\pa_t\vecm_{h,k})}
{\vecm_{h,k}\times\vecpsi}_{\mL^2(D_T)}
\right|,\\
&\leq
2\norm{\vecm_{h,k}^--\vecm_{h,k}}{\mL^2(D_T)}
\norm{\vecv_{h,k}}{\mL^2(D_T)}
\norm{\vecm_{h,k}^-}{\mL^{\infty}(D_T)}
\norm{\vecpsi}{\mL^{\infty}(D_T)}\\
&\quad+
\norm{\vecv_{h,k}-\pa_t\vecm_{h,k}}{\mL^1(D_T)}
\norm{\vecm_{h,k}^-}{\mL^{\infty}(D_T)}
\norm{\vecpsi}{\mL^{\infty}(D_T)}.
\end{align*}
Therefore, the bound of $I_1$ can be obtained by using Lemmas~\ref{lem:3.2a} and~\ref{lem:3.4}. The bounds for $I_2, I_3$ and $I_4$  can be carried out similarly. This completes the proof of the lemma.
\end{proof}
In order to prove the convergence of random variables $\vecm_{h,k}$, we first show  in the following lemma that the family $\cL(\vecm_{h,k})$ is tight.
\begin{lemma}\label{lem:tig}
Assume that $h$ and $k$ go to 0 satisfying~\eqref{equ:theta}.
Then the set of laws~$\{\cL(\vecm_{h,k})\}$ on the Banach space
$C\big([0,T];\mH^{-1}(D)\big)$ is tight.
\end{lemma}
\begin{proof}
For $r\in\R^+$, we define
\[
B_r :=
\{\vecu\in H^1(0,T;\mL^2(D))\, :\,\norm{\vecu}{H^1(0,T;\mL^2(D))}\leq r\}.
\]
Firstly, from the definition of $\cL(\vecm_{h,k})$ we have
\begin{align*}
\cL(\vecm_{h,k})(B_r)
&=
\mP\{\omega\in\Omega\,:\,\vecm_{h,k}(\omega)\in B_r \}
=
1-\mP\{\omega\in\Omega\,:\,\vecm_{h,k}(\omega)\in B_r^c \},
\end{align*}
where $B_r^c$ is the complement of $B_r$ in $H^1(0,T;\mL^2(D))$.
Furthermore, from the definition of $B_r$ and~\eqref{equ:mhk h1}, we deduce
\begin{align*}
\cL(\vecm_{h,k})(B_r)
&\geq
1
-\frac{1}{r^2}
\int_{\Omega} \norm{\vecm_{h,k}(\omega)}{H^1(0,T;\mL^2(D))}^2 \mP\, d\omega\\
&\geq
1-\frac{c}{r^2}.
\end{align*}
By Lemma~\ref{lem:imbed}, $H^1(0,T;\mL^2(D))$ is compactly imbedded in $C\big([0,T];\mH^{-1}(D)\big)$. This allows us to conclude that the family of laws~$\{\cL(\vecm_{h,k})\}$ is tight on $C\big([0,T];\mH^{-1}(D)\big)$.
\end{proof}
From~\eqref{Def:W}, the approximation $W_k$ of the Wiener process
$W$ belongs to~$\D(0,T)$, the so-called Skorokhod space.
We recall that the set of laws $\{\cL(W_k)\}$ is tight on $\D(0,T)$;
see e.g.~\cite{Bill99}.
The following proposition is a consequence of the tightness of
$\{\cL(\vecm_{h,k})\}$ and  $\{\cL(W_k)\}$.
\begin{proposition}\label{pro:con}
Assume that $h$ and $k$ go to 0 satisfying~\eqref{equ:theta}.
Then there exist
\begin{enumerate}
\renewcommand{\labelenumi}{(\alph{enumi})}
\item
a probability space $(\Omega',\cF',\mP')$,
\item
a sequence $\{(\vecm'_{h,k},W_k')\}$ of random variables
defined on $(\Omega',\cF',\mP')$ and taking values in the space $C\big([0,T];\mH^{-1}(D)\big)\times\D(0,T)$, 
\item
a random variable $(\vecm',W')$ defined on
$(\Omega',\cF',\mP')$ and taking values in $C\big([0,T];\mH^{-1}(D)\big)\times
\D(0,T)$.
\end{enumerate}
satisfying
\begin{enumerate}
\item\label{item:a}
$\cL(\vecm_{h,k},W_k) = \cL(\vecm_{h,k}',W_k')$,
\item\label{item:b}
$\vecm_{h,k}'\goto\vecm'$ in $C\big([0,T];\mH^{-1}(D)\big)$ strongly,
$\mP'$-a.s.,
\item\label{item:c}
$W_k'\goto W'$ in $\D(0,T)$ $\mP'$-a.s.
\end{enumerate}

Moreover, the sequence $\{\vecm'_{h,k}\}$ satisfies $\mathbb P^\prime$-a.s. 
\begin{align}
\norm{\vecm_{h,k}'(\omega^\prime)}{\mH^1(D_T)}
&\leq c, \label{equ:mhk' h1} \\
\norm{\vecm_{h,k}'(\omega^\prime)}{\mL^{\infty}(D_T)}
&\leq c, 
\label{mhk'Linf} \\
\norm{|\vecm_{h,k}'(\omega^\prime)|-1}{\mL^2(D_T)} &\leq c(h+k).\label{equ:mhk'1}
\end{align}
\end{proposition}
\begin{proof}
By Lemma~\ref{lem:tig} and the Donsker Theorem \cite[Theorem 8.2]{Bill99}
the family of probability measures $\{\cL(\vecm_{h,k},W_{k})\}$ is tight on  
$C\big([0,T];\mH^{-1}(D)\big)\times\D(0,T)$. Then by Theorem 5.1 in
\cite{Bill99} the family of measures $\{\cL(\vecm_{h,k},W_{k})\}$ is
relatively compact on  
$C\big([0,T];\mH^{-1}(D)\big)\times\D(0,T)$, that is there exists a
subsequence, still denoted by 
$\{\cL(\vecm_{h,k},W_{k})\}$, such that $\{\cL(\vecm_{h,k},W_{k})\}$
converges weakly. 
Hence, the existence of~(a)--(c)
satisfying~\eqref{item:a}--\eqref{item:c} 
follows immediately from the Skorokhod Theorem~\cite[Theorem 6.7]{Bill99}
since $C\big([0,T];\mH^{-1}(D)\big)\times \D(0,T)$ is a separable metric
space.

The 
estimates~\eqref{equ:mhk' h1}--\eqref{equ:mhk'1} 
are direct consequences of the equality of laws 
\[
\cL(\vecm_{h,k}) = \cL(\vecm_{h,k}')
\quad\text{in}\quad 
C\big([0,T];\mH^{-1}(D)\big)
\] 
stated in part (1) of the
proposition and the fact that a
Borel set in $\mathbb H^1(D_T)$ or $\mathbb H^1(D_T) \cap \mL^\infty(D_T)$ can be identified with a Borel set
in $C\big([0,T];\mH^{-1}(D)\big)$. 

Indeed,  let $B_c$ stand for the centered ball of radius $c$ in
$\mathbb H^1\left(D_T\right)$. Then \eqref{equ:mhk h1} yields
\[
\mathbb P\left(\left\{\omega\in\Omega:\,\vecm_{h,k}(\omega)\in B_c\right\}\right)
=
1,
\]
implying
\[
\mP^\prime
\left(\left\{\omega^\prime\in\Omega^\prime:\,\vecm_{h,k}
\left(\omega^\prime\right)\in B_c\right\}\right)=1,
\]
which is equivalent to ~\eqref{equ:mhk' h1}.
The remaining estimates can be justified in exactly the same way, where the
Borel set $B_c$ is chosen to be
\[
B_c
=
\{\vecu\in\mH^1(D_T) \cap \mL^{\infty}(D_T) 
: \norm{\vecu}{\mL^{\infty}(D_T)} \leq c\} 
\]
in the case of~\eqref{mhk'Linf}, and
\[
B_c
=
\{\vecu\in\mH^1(D_T) : \norm{|\vecu|-1}{\mL^2(D_T)} \leq c(h+k)\}
\]
in the case of~\eqref{equ:mhk'1}.
\end{proof}

We are now ready to state and prove the main theorem.
\begin{theorem}\label{the:mai}
Assume that $T>0$, $\vecM_0\in\mH^2(D)$ and $\vecg\in\mW^{2,\infty}(D)$
satisfy~\eqref{equ:m0} and \eqref{equ:g 1}, respectively. 
Then $\vecm'$, $W'$, the sequence $\{\vecm'_{h,k}\}$ and the probability space $(\Omega',\cF',\mP')$  given by Proposition~\ref{pro:con} satisfy
\begin{enumerate}
\item\label{ite:the1}
the sequence $\{\vecm'_{h,k}\}$ converges to $\vecm'$
weakly in $\mH^1(D_T)$, $\mP'$-a.s.
\item\label{ite:the2}
$\big(\Omega',\cF',(\cF'_t)_{t\in[0,T]},\mP',W',\vecM'\big)$ is a weak
martingale solution of~\eqref{E:1.1},
where 
\[
\vecM'(t):=e^{W'(t)G}\vecm'(t)\quad \forall t\in[0,T],\text{ a.e. } \vecx\in D.
\] 
\end{enumerate}
\end{theorem}
\begin{proof}
By Proposition~\ref{pro:con} 
there exists a set
$V\subset\Omega^\prime$ such that $\mathbb P^\prime(V)=1$,
\begin{equation}\label{equ:mphk mp}
\vecm^\prime_{h,k}\left(\omega^\prime\right) \to
\vecm^\prime\left(\omega^\prime\right)
\quad\text{in}\quad C\left([0,T];\mathbb H^{-1}(D)\right),
\quad\forall\omega'\in V,
\end{equation}
and \eqref{equ:mhk' h1}
holds for every $\omega^\prime\in V$. In what follows, we 
work with a fixed $\omega^\prime\in V$. In view of
~\eqref{equ:mhk' h1}
we can use the Banach--Alaoglu Theorem to deduce that there exists a
subsequence of $\left\{\vecm_{h,k}'(\omega')\right\}$ converging in
$\mathbb H^1(D_T)$ weakly to $\vecm'(\omega')$. 
If the sequence $\left\{\vecm_{h,k}'(\omega')\right\}$ has another
point of accumulation in the weak $\mH^1(D_T)$ topology, 
apart from $\vecm'(\omega')$, then this contradicts~\eqref{equ:mphk
mp}. Hence the whole sequence $\left\{\vecm_{h,k}'(\omega')\right\}$
converges to $\vecm'(\omega')$ weakly in $\mH^1(D_T)$. Repeating this argument
for all $\omega'\in V$ proves Part~\eqref{ite:the1} of the theorem.
 
In order to prove Part~\eqref{ite:the2}, by noting Lemma~\ref{lem:equi}
we only need to prove that $\vecm'$ and $W'$
satisfy~\eqref{equ:m 1} and~\eqref{InE:13}, namely
\begin{equation}\label{m'eq1}
|\vecm'(t,\vecx,\omega^\prime)| = 1,\quad t\in (0,T),\quad \vecx\in D, 
\end{equation}
and
\begin{align}\label{InE:13p}
\lambda_1\inpro{\vecm'_t}{\vecvarphi}_{\mL^2(D_T)}
+
\lambda_2\inpro{\vecm'\times\vecm'_t}{\vecvarphi}_{\mL^2(D_T)} 
&=
\mu \inpro{\nabla\vecm'}{\nabla(\vecm'\times\vecvarphi)}_{\mL^2(D_T)}
+
\lambda_1\inpro{F(t,\vecm')}{\vecvarphi}_{\mL^2(D_T)}
\nn\\
&+
\lambda_2
\inpro{\vecm'\times F(t,\vecm')}{\vecvarphi}_{\mL^2(D_T)} 
\quad\forall \vecvarphi\in L^2(0,T;\mH^1(D)).
\end{align}

Since $\mH^1(D_T)$ is compactly embedded in $\mL^2(D_T)$, there exists a subsequence of $\{\vecm_{h,k}'(\omega^\prime)\}$ (still denoted by $\{\vecm_{h,k}'(\omega^\prime)\}$) such that 
\begin{equation}\label{strongconmhk'}
\{\vecm_{h,k}'(\omega^\prime)\}
\rightarrow 
\vecm'(\omega^\prime)\text{ in }\mL^2(D_T) \text{ strongly}.
\end{equation}
Therefore~\eqref{m'eq1} follows from~\eqref{strongconmhk'} 
and~\eqref{equ:mhk'1}.

In order to prove~\eqref{InE:13p} we first find the equation satisfied by
$\vecm'_{h,k}$ and $W_k'$ and then pass to the limit when $h,k\to0$.

For any $\vecpsi\in C_0^\infty\big((0,T);\C^{\infty}(D)\big)$,
putting the left-hand side of~\eqref{InE:10} by 
$\cI_{\vecpsi}(\vecm_{h,k},W_{k})$, namely,
\begin{align*}
\cI_{\vecpsi}(\vecm_{h,k},W_{k})
&=
-\lambda_1
\inpro{\vecm_{h,k}\times\pa_t\vecm_{h,k}}
{\vecm_{h,k}\times\vecpsi}_{\mL^2(D_T)}
+
\lambda_2\inpro{\pa_t\vecm_{h,k}}
{\vecm_{h,k}\times\vecpsi}_{\mL^2(D_T)}
\nonumber\\
&\quad +
\mu\inpro{\nabla\vecm_{h,k}}
{\nabla(\vecm_{h,k}\times\vecpsi)}_{\mL^2(D_T)}
+
\inpro{R_{h,k}(\cdot,\vecm_{h,k})}
{\vecm_{h,k}\times\vecpsi}_{\mL^2(D_T)},
\end{align*}
it follows from~\eqref{InE:10} that for $h$, $k$ sufficiently small we have
\[
\mP(
\{\omega\in\Omega : \cI_{\vecpsi}(\vecm_{h,k},W_{k}) = O(hk) 
\} = 1.
\]
By using the same argument as in the proof of Propostion~\ref{pro:con} we
deduce
\[
\mP'(
\{\omega^\prime\in\Omega' : \cI_{\vecpsi}(\vecm'_{h,k},W'_{k}) = O(hk) 
\} = 1,
\]
or equivalently
\begin{align}\label{equ:mhk'}
-\lambda_1&\inpro{\vecm_{h,k}'(\omega^\prime)\times\pa_t\vecm_{h,k}'(\omega^\prime)}
{\vecm_{h,k}'(\omega^\prime)\times\vecpsi}_{\mL^2(D_T)}
+
\lambda_2\inpro{\pa_t\vecm_{h,k}'(\omega^\prime)}
{\vecm_{h,k}'(\omega^\prime)\times\vecpsi}_{\mL^2(D_T)}
\nonumber\\
&+
\mu\inpro{\nabla\vecm_{h,k}'(\omega^\prime)}
{\nabla(\vecm_{h,k}'(\omega^\prime)\times\vecpsi)}_{\mL^2(D_T)} \nn\\
&+
\inpro{R_{h,k}(\cdot,\vecm_{h,k}'(\omega^\prime))}
{\vecm_{h,k}'(\omega^\prime)\times\vecpsi}_{\mL^2(D_T)}
= O(hk),\quad \forall\vecpsi \in C_0^\infty\big((0,T);\C^{\infty}(D)\big).
\end{align}

It suffices now to use the same arguments as in \cite[Theorem~4.5]{LeTra12} to
pass to the limit in~\eqref{equ:mhk'}.
For the convenience of the reader, we reproduce the proof here.

We prove that as $h$ and $k$ tend to $0$ there hold 
\begin{align}
\inpro{\vecm_{h,k}'(\omega^\prime)\times\pa_t\vecm_{h,k}'(\omega^\prime)}
{\vecm_{h,k}'(\omega^\prime)\times\vecpsi}_{\mL^2(D_T)}
&\rightarrow
\inpro{\vecm'(\omega^\prime)\times\pa_t\vecm'(\omega^\prime)}
{\vecm'(\omega^\prime)\times\vecpsi}_{\mL^2(D_T)},\label{convergent1}\\
\inpro{\pa_t\vecm_{h,k}'(\omega^\prime)}
{\vecm_{h,k}'(\omega^\prime)\times\vecpsi}_{\mL^2(D_T)}
&\rightarrow
\inpro{\pa_t\vecm'(\omega^\prime)}
{\vecm'(\omega^\prime)\times\vecpsi}_{\mL^2(D_T)},\label{convergent2}\\
\inpro{\nabla\vecm_{h,k}'(\omega^\prime)}
{\nabla(\vecm_{h,k}'(\omega^\prime)\times\vecpsi)}_{\mL^2(D_T)}
&\rightarrow
\inpro{\nabla\vecm'(\omega^\prime)}
{\nabla(\vecm'(\omega^\prime)\times\vecpsi)}_{\mL^2(D_T)},\label{convergent3}\\
\inpro{R_{h,k}(\cdot,\vecm_{h,k}'(\omega^\prime))}
{\vecm_{h,k}'(\omega^\prime)\times\vecpsi}_{\mL^2(D_T)}
&\rightarrow
\inpro{R(\cdot,\vecm'(\omega^\prime))}
{\vecm'(\omega^\prime)\times\vecpsi}_{\mL^2(D_T)}.\label{convergent4}
\end{align}

To prove~\eqref{convergent1} we use
the triangle inequality and Holder's inequality to obtain
\begin{align*}
\cI &:=
\big|\inpro{\vecm_{h,k}'(\omega^\prime)\times\pa_t\vecm_{h,k}'(\omega^\prime)}
{\vecm_{h,k}'(\omega^\prime)\times\vecpsi}_{\mL^2(D_T)}
-
\inpro{\vecm'(\omega^\prime)\times\pa_t\vecm'(\omega^\prime)}
{\vecm'(\omega^\prime)\times\vecpsi}_{\mL^2(D_T)}\big| \\
&\leq
\left|\inpro{\vecm_{h,k}'(\omega^\prime)\times\pa_t\vecm_{h,k}'(\omega^\prime)}
{\big(\vecm_{h,k}'(\omega^\prime)-\vecm'(\omega^\prime)\big)\times\vecpsi}_{\mL^2(D_T)}\right|\\
&\quad +
\left|\inpro{\big(\vecm_{h,k}'(\omega^\prime)-\vecm'(\omega^\prime)\big)\times\pa_t\vecm_{h,k}'(\omega^\prime)}
{\vecm'(\omega^\prime)\times\vecpsi}_{\mL^2(D_T)}\right|\\
&\quad +
\left|\inpro{\vecm'(\omega^\prime)\times\big(\pa_t\vecm_{h,k}'(\omega^\prime)-\pa_t\vecm'(\omega^\prime)\big)}
{\vecm'(\omega^\prime)\times\vecpsi}_{\mL^2(D_T)}\right|\\
&\leq
\norm{\vecm_{h,k}'(\omega^\prime)}{\mL^{\infty}(D_T)}
\norm{\vecpsi}{\mL^{\infty}(D_T)}
\norm{\pa_t\vecm_{h,k}'(\omega^\prime)}{\mL^2(D_T)}
\norm{\vecm_{h,k}'(\omega^\prime)-\vecm'(\omega^\prime)}{\mL^2(D_T)}\\
&\quad +
\norm{\vecm'(\omega^\prime)}{\mL^{\infty}(D_T)}
\norm{\vecpsi}{\mL^{\infty}(D_T)}
\norm{\pa_t\vecm_{h,k}'(\omega^\prime)}{\mL^2(D_T)}
\norm{\vecm_{h,k}'(\omega^\prime)-\vecm'(\omega^\prime)}{\mL^2(D_T)}\\
&\quad +
\left|\inpro{\pa_t\vecm_{h,k}'(\omega^\prime)-\pa_t\vecm'(\omega^\prime)}
{\big(\vecm'(\omega^\prime)\times\vecpsi\big)\times\vecm'(\omega^\prime)}_{\mL^2(D_T)}\right|,
\end{align*}
where in the last step we used~\eqref{equ:ele ide}.
It follows from~\eqref{equ:mhk' h1} and~\eqref{mhk'Linf} that
\[
\cI
\le
c 
\norm{\vecm_{h,k}'(\omega^\prime)-\vecm'(\omega^\prime)}{\mL^2(D_T)}
+
\left|
\inpro{\pa_t\vecm_{h,k}'(\omega^\prime)-\pa_t\vecm'(\omega^\prime)}
{\big(\vecm'(\omega^\prime)\times\vecpsi\big)\times\vecm'(\omega^\prime)}_{\mL^2(D_T)}
\right|.
\]
Hence~\eqref{convergent1} follows from Part~\eqref{ite:the1}
and~\eqref{strongconmhk'}.

Statements~\eqref{convergent2} and~\eqref{convergent4} can be proved in the
same manner.
To prove~\eqref{convergent3} we first note that for any 
vector functions $\vecu$ and $\vecpsi$ there holds
\[
\nabla\vecu \cdot \nabla(\vecu\times\vecpsi)
=
\nabla\vecu \cdot (\vecu \times \nabla\vecpsi),
\]
provided that the gradients exist (at least in the weak sense).
Therefore, \eqref{convergent3} is equivalent to
\begin{equation*}
\inpro{\nabla\vecm_{h,k}'(\omega^\prime)}
{\vecm_{h,k}'(\omega^\prime)\times\nabla\vecpsi}_{\mL^2(D_T)}
\rightarrow
\inpro{\nabla\vecm'(\omega)}
{\vecm'(\omega^\prime)\times\nabla\vecpsi}_{\mL^2(D_T)}.\label{convergent3'}
\end{equation*}
The above statement can be proved in the same manner as that
of~\eqref{convergent1}.

Equation~\eqref{InE:13p} now follows
from~\eqref{equ:mhk'}--\eqref{convergent4}, completing the proof of
the theorem.
\end{proof}
\section{Numerical experiments}\label{sec:num}
In this section we solve an academic example of the stochastic
LLG equation which is studied in~\cite{BanBrzPro13,BanBrzPro09}.

The computational  domain $D$ is the unit square
$D=(-0.5,0.5)^2$, the given function $\vecg= (1,0,0)$ is constant,
and the initial condition $\vecM_0$ is defined below:
\begin{equation*}
\vecM_0(\vecx)
=
\begin{cases}
(2\vecx^*A,A^2-|\vecx^*|^2)/(A^2+|\vecx^*|^2), \quad &|\vecx^*|< \frac{1}{2},\\
(-2\vecx^*A,A^2-|\vecx^*|^2)/(A^2+|\vecx^*|^2), \quad\frac{1}{2}\leq&|\vecx^*|\leq 1,\\
(-\vecx^*,0)/|\vecx^*|,\quad &|\vecx^*| \geq 1,
\end{cases}
\end{equation*}
where $\vecx^* = 2\vecx$ and $A= (1-2|\vecx^*|)^4$.
From~\eqref{equ:vecm},~\eqref{equ:G8} and noting that $W(0)=0$,
we have $\vecm(0,\cdot)=\vecM(0,\cdot)$. We set the values for
the other parameters in~\eqref{E:1.1} as $\lambda_1=\lambda_2=1$
and the parameter $\theta$ in Algorithm~\ref{Algo:1} is chosen
to be $0.7$.

For each time step $k$, we generate a discrete  Brownian path by:
\[
W_k(t_{j+1})-W_k(t_{j})\sim \cN(0,k)\quad \text{for all
}j=0,\cdots,J-1.
\]
An approximation of any expected value is computed as
the average of $L$
discrete Brownian paths. In our experiments, we choose  $L=400$.

Having computed $\vecm_h^{(j)}$ by using
Algorithm~\ref{Algo:1}, we compute $\vecM^{(j)}_h =
e^{W_k(t_j)G_h}(\vecm^j_h)$ by using~\eqref{equ:G8}.
The discrete solutions $\vecM_{h,k}$ and $\vecM_{h,k}^-$
of~\eqref{E:1.1}
are defined as in Definition~\ref{def:mhk} with
$\vecm_h^{(j)}$ replaced by $\vecM^{(j)}_h$.

In the first set of experiments, to observe
convergence of the method,
we solve with $T=1$, $h=1/n$ where $n=10,20,30,40,50$,
and different time steps $k=h$, $k=h/2$, and $k=h/4$.
For each value of $h$,
the domain $D$ is uniformly partitioned into triangles
of size $h$.

Noting that
\begin{align*}
E_{h,k}^2
&:=\mE\left(
\int_{D_T}\left|1-|\vecM_{h,k}^-|\right|^2\dvx\dt\right)
=
\mE\left(\norm{|\vecM|-|\vecM_{h,k}^-|}{\mL^2(D_T)}^2\right)\\
&\le
\mE\left(\norm{\vecM-\vecM_{h,k}^-}{\mL^2(D_T)}^2\right),
\end{align*}
we compute and plot in
Figure~\ref{fig:error} the error $E_{h,k}$ for different values of $h$ and
$k$. The figure suggests a clear convergence of the method.
\begin{figure}
\begin{center}
\includegraphics[width=1.0\textwidth,height=0.5\textheight]{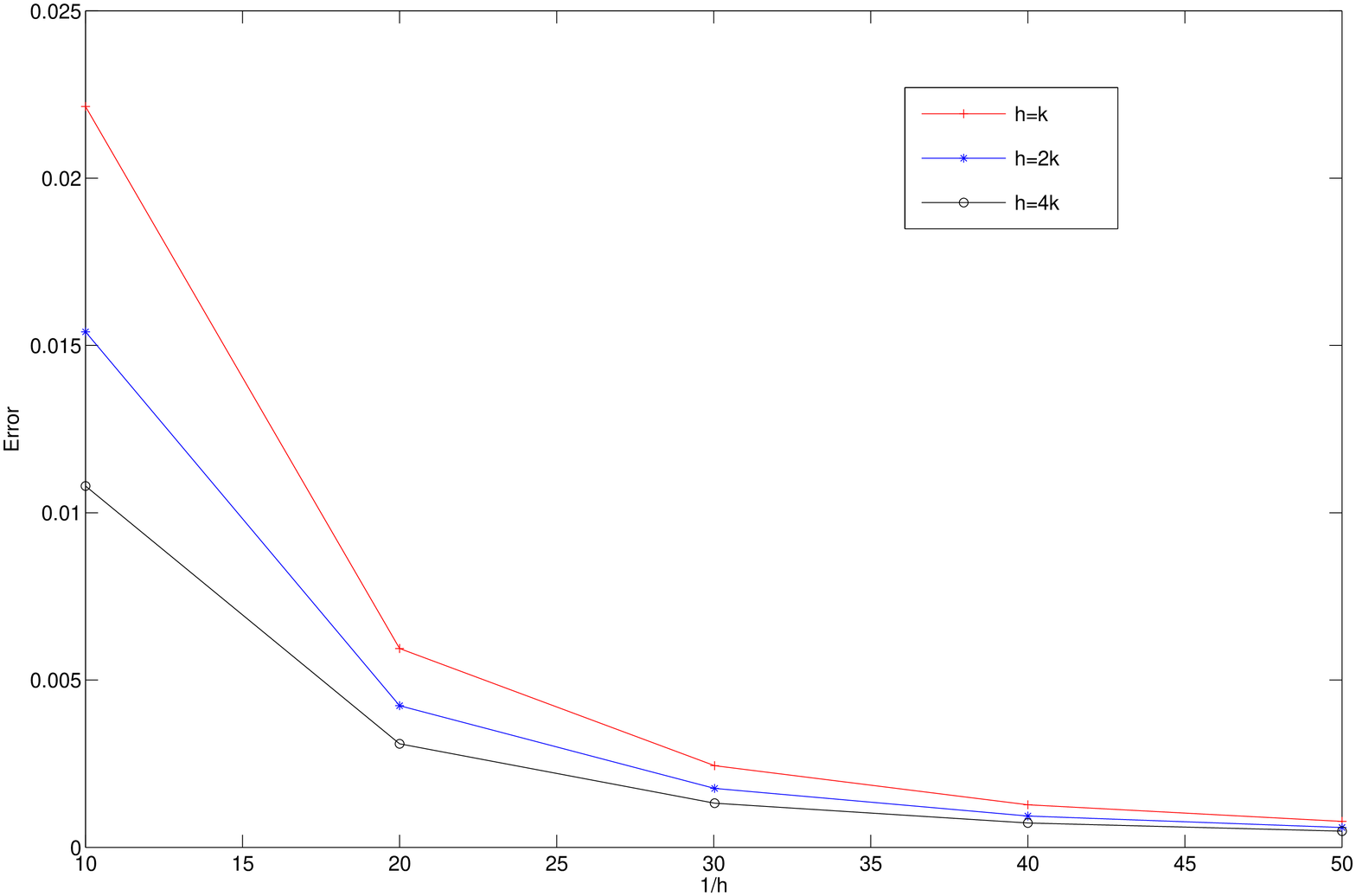}
\caption{Plot of error $E_{h,k}$, $1/h\longmapsto E_{h,k}$}
\label{fig:error}
\end{center}
\end{figure}

In the second set of experiments in order to observe boundedness
of the discrete energy
$t\mapsto\mE\left(\|\nabla\vecM_{h,k}(t)\|_{\mL^2(D)}^2\right)$,
we solve the problem with fixed
values of $h$ and $k$, namely
$h=1/60$ and $k=1/100$. In Figure~\ref{fig:energy} we
plot this energy for
different values of $\lambda_2$. The graphs suggest
that the energy approaches 0 when $t\goto\infty$.
\begin{figure}
\begin{center}
\includegraphics[width=1.0\textwidth,height=0.5\textheight]{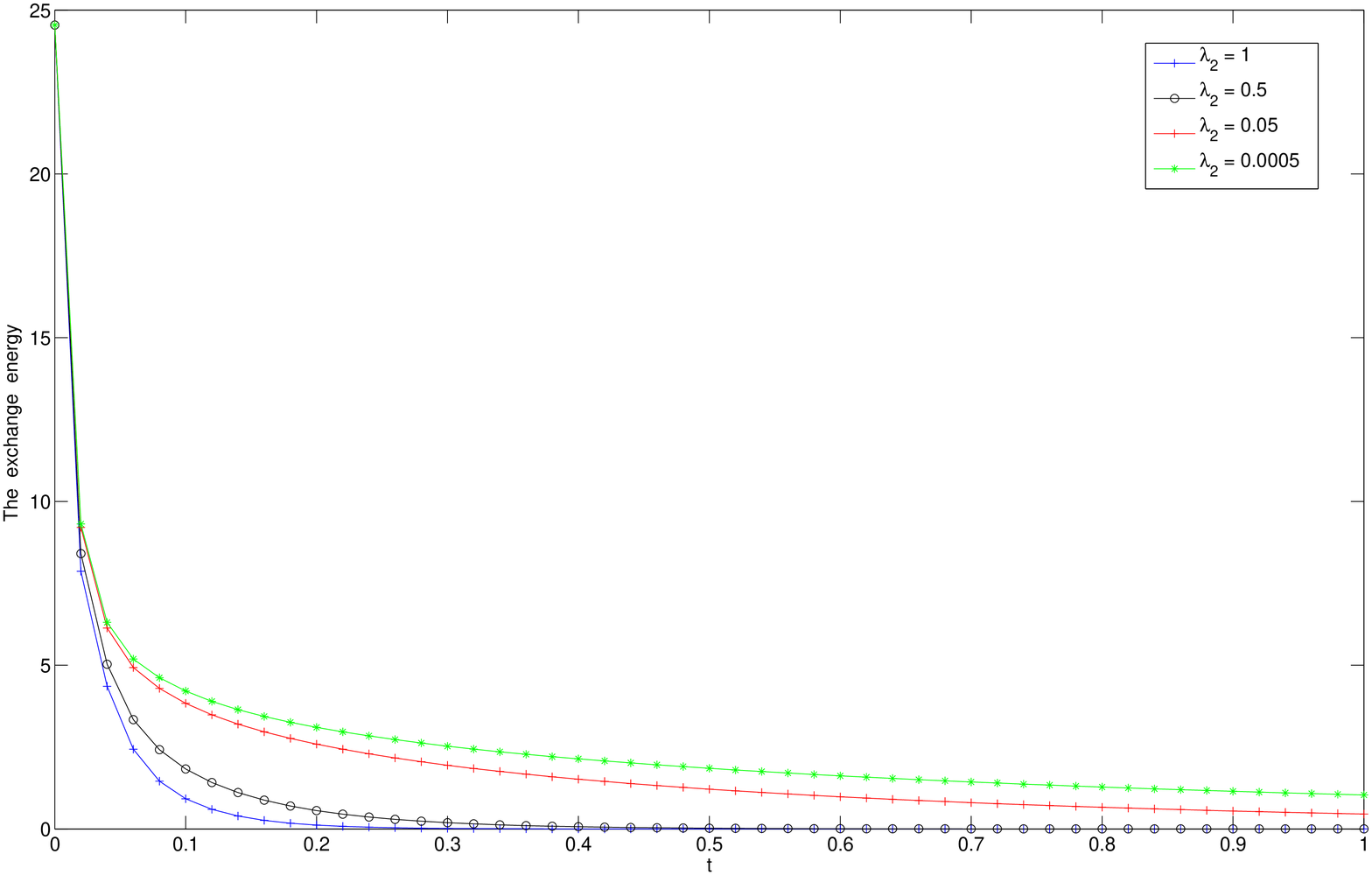}
\caption{Plot of the exchange energy, $t\longmapsto\mE\left(\|\nabla\vecM_{h,k}(t)\|_{\mL^2(D)}^2\right)$}
\label{fig:energy}
\end{center}
\end{figure}

Finally, in Figure~\ref{fig:magne} we plot snapshots of the
magnetization vector field $\mE\left(\vecM_{h,k}\right)$
at different time levels, where $h=1/50$ and $k=1/80$.
These vectors are coloured according to their magnitudes.
\begin{figure}
\begin{center}
\includegraphics[width=1.2\textwidth,height=0.5\textheight]{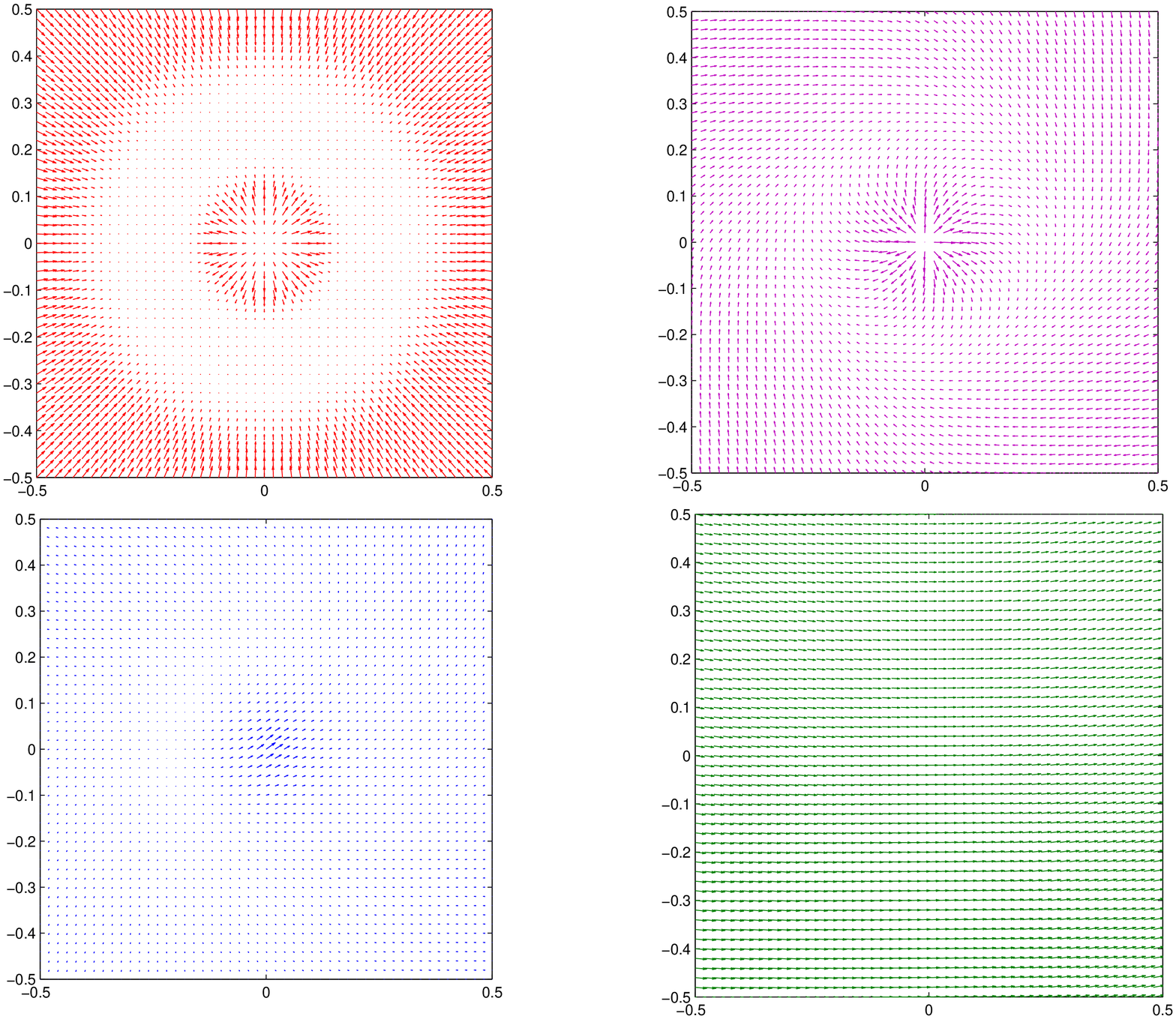}
\caption{Plot of magnetizations $\mE\left(\vecM_{h,k}(t,x)\right)$ at $t$ $= 0$,  $0.0625$, $0.3125$, $0.4375$. Vectors are coloured according to the value of $|\mE\left(\vecM_{h,k}\right)|$
(red: value $= 1$, pink: value $\approx0.98$, blue: value $\approx0.87$, green: value $\approx0.82$)}
\label{fig:magne}
\end{center}
\end{figure}

A comparision of our method and the method proposed
in~\cite{BanBrzPro13} is presented in Table~\ref{Tab:Compe}.

\begin{table}[h]
\caption{A comparision of the BBP
(Ba\v{n}as, Bartels and Prohl) method and our method}
\label{Tab:Compe}
\begin{center}
\begin{tabular}{|c| c| c|}
\hline
\quad & BBP method & Our method\\
\hline
The discrete system & nonlinear & linear\\
\hline
Degrees of freedom & $3N$& $2N$\\
\hline
Change of basis functions & NO & YES \\
at each iteration & & \\
\hline
Systems to be solved & & \\
at each iteration & $L$ & 1 \\
when $\vecg = $ const. & & \\
\hline
\end{tabular}

\end{center}
\end{table}

\section{Appendix}\label{sec:app}
\begin{lemma}\label{lem:4.0}
For any real constants $\lambda_1$ and $\lambda_2$ with
$\lambda_1\not=0$, if $\vecpsi, \veczeta\in\R^3$
satisfy $|\veczeta|=1$, then there exists
$\vecvarphi\in\R^3$ satisfying
\begin{equation}\label{equ:app}
\lambda_1\vecvarphi
+
\lambda_2\vecvarphi\times\veczeta
=\vecpsi.
\end{equation}
As a consequence, if
$\veczeta\in\mH^1(D_T)$ with $|\veczeta(t,x)|=1$ a.e. in
$D_T$ and $\vecpsi\in L^2(0,T;\mW^{1,\infty}(D))$, then
$\vecvarphi\in L^2(0,T;\mH^1(D))$.
\end{lemma}
\begin{proof}
It is easy to see that \eqref{equ:app} is equivalent to
the linear system
\[
A\vecvarphi=\vecpsi
\]
where
\[
A=
\left(
\begin{matrix}
\lambda_1      & \lambda_2 \zeta_3  & -\lambda_2 \zeta_2\\
-\lambda_2 \zeta_3 & \lambda_1      & \lambda_2 \zeta_1\\
\lambda_2 \zeta_2  & -\lambda_2 \zeta_1 & \lambda_1
\end{matrix}
\right)
\]
and $\veczeta=(\zeta_1,\zeta_2,\zeta_3)$.
It follows from the condition~$|\veczeta|=1$ that
$\det(A)=\lambda_1(\lambda_1^2+\lambda_2^2)\not=0$,
which implies the existence of $\vecvarphi$. The fact
that $\vecvarphi\in L^2(0,T;\mH^1(D))$ when
$\veczeta\in\mH^1(D_T)$ and
$\vecpsi\in L^2(0,T;\mW^{1,\infty}(D))$ can be easily checked.
\end{proof}

\begin{lemma}\label{lem:Ih vh}
For any $\vecv\in\C(D)$, $\vecv_h\in\mV_h$ and
$\vecpsi\in\C_0^\infty(D_T)$,
\begin{align*}
\norm{I_{\mV_h}\vecv}{\mL^{\infty}(D)}
&\le
\norm{\vecv}{\mL^{\infty}(D)}, \\
\norm{\vecm_{h,k}^-\times\vecpsi
-
I_{\mV_h}(\vecm_{h,k}^-\times\vecpsi)}{\mL([0,T],\mH^1(D))}^2
&\le
ch^2
\norm{\vecm_{h,k}^-}{\mL([0,T],\mH^1(D))}^2
\norm{\vecpsi}{\mW^{2,\infty}(D_T)}^2,
\end{align*}
where $\vecm_{h,k}^-$ is  defined in Defintion~\ref{def:mhk}
\end{lemma}
\begin{proof}
We note that for any $\vecx\in D$ there are atmost 4 basis
functions $\phi_{n_i}$, $i=1,\ldots,4$, being nonzero at
$\vecx$. Moreover, $\sum_{i=1}^4 \phi_{n_i}(\vecx) = 1$.
Hence
\[
\snorm{I_{\mV_h}\vecv(\vecx)}{}
=
\snorm{\sum_{i=1}^4 \vecv(\vecx_{n_i}) \phi_{n_i}(\vecx)}{}
\le
\norm{\vecv}{\mL^{\infty}(D)}.
\]
The proof for the second inequality can be done by using the interpolation error (see e.g.~\cite{Johnson87}) and the linearity of $\vecm_{h,k}^-$ on each triangle $K$, as follows
\begin{align*}
\norm{\vecm_{h,k}^-\times\vecpsi
-
I_{\mV_h}(\vecm_{h,k}^-\times\vecpsi)}{\mH^1(K)}^2
&\leq
ch^2
\norm{\nabla^2\big(\vecm_{h,k}^-\times\vecpsi\big)}{K}^2\\
&\leq
ch^2
\norm{\vecm_{h,k}^-}{\mH^1(K)}^2
\norm{\vecpsi}{\mW^{2,\infty}(K)}^2.
\end{align*}
We now obtain the second inequality by summing over all the triangles of $\cT_h$ and integrating in time the above inequality, which completes the proof.
\end{proof}

The next lemma defines a discrete $\mL^p$-norm in
$\mV_h$ which is equivalent to the usual $\mL^p$-norm.
\begin{lemma}\label{lem:nor equ}
There exist $h$-independent positive constants $C_1$ and
$C_2$ such that for all $p\in[1,\infty]$ and
$\vecu\in\mV_h$,
\begin{equation*}
C_1\|\vecu\|^p_{\mL^p(\Omega)}
\leq
h^d
\sum_{n=1}^N |\vecu(\vecx_n)|^p
\leq
C_2\|\vecu\|^p_{\mL^p(\Omega)},
\end{equation*}
where $\Omega\subset\R^d$, d=1,2,3.
\end{lemma}
\begin{proof}
A proof of this lemma for $p=2$ and $d=2$ can be found
in~\cite[Lemma 7.3]{Johnson87} or~\cite[Lemma
1.12]{Chen05}. The result for general values of $p$ and $d$
can be obtained in the same manner.
\end{proof}
\begin{lemma}\label{lem:imbed}
 $H^1(0,T;\mL^2(D))$ is compactly imbedded in $C\big([0,T];\mH^{-1}(D)\big)$.
\end{lemma}
\begin{proof}
We note that $\mL^2(D)$ is compactly imbedded in $\mH^{-1}(D)$.
By using~\cite[Theorem 2.2]{Flan95}, we deduce that the embedding $H^1(0,T;\mL^2(D))\hookrightarrow C\big([0,T];\mH^{-1}(D)\big)$ is compact.
\end{proof}

\section*{Acknowledgements}
The authors acknowledge financial support through the ARC project DP120101886.
They thank the anonymous referee for pointing out a pitfall in the
manuscript and for his constructive criticisms which help to improve
the presentation of the paper.

\bibliographystyle{myabbrv}
\bibliography{mybib}

\def\cprime{$'$}
\begin{thebibliography}{10}

\bibitem{Alo08}
F.~Alouges.
\newblock A new finite element scheme for {L}andau-{L}ifchitz equations.
\newblock {\em Discrete Contin. Dyn. Syst. Ser. S},  {\bf 1} (2008), 187--196.

\bibitem{Alo14}
F.~Alouges, A.~D. Bouard, and A.~Hocquet.
\newblock A semi-discrete scheme for the stochastic {L}andau--{L}ifshitz
  equation.
\newblock {R}esearch {R}eport, arXiv:1403.3016, 2014.

\bibitem{AloJai06}
F.~Alouges and P.~Jaisson.
\newblock Convergence of a finite element discretization for the
  {L}andau-{L}ifshitz equations in micromagnetism.
\newblock {\em Math. Models Methods Appl. Sci.},  {\bf 16} (2006), 299--316.

\bibitem{AloSoy92}
F.~Alouges and A.~Soyeur.
\newblock On global weak solutions for {L}andau-{L}ifshitz equations:
  {E}xistence and nonuniqueness.
\newblock {\em Nonlinear Anal.},  {\bf 18} (1992), 1071--1084.

\bibitem{BanBrzNekPro13}
{\v{L}}.~Ba{\v{n}}as, Z.~Brze{\'z}niak, M.~Neklyudov, and A.~Prohl.
\newblock {\em Stochastic Ferromagnetism -- Analysis and Numerics}.
\newblock de Gruyter Series in Mathematics, 58. de Gruyter, 2013.

\bibitem{BanBrzPro13}
{\v{L}}.~Ba{\v{n}}as, Z.~Brze{\'z}niak, and A.~Prohl.
\newblock Computational studies for the stochastic
  {L}andau--{L}ifshitz--{G}ilbert equation.
\newblock {\em SIAM J. Sci. Comput.},  {\bf 35} (2013), B62--B81.

\bibitem{Bart05}
S.~Bartels.
\newblock Stability and convergence of finite-element approximation schemes for
  harmonic maps.
\newblock {\em SIAM J. Numer. Anal.},  {\bf 43} (2005), 220--238 (electronic).

\bibitem{BanBrzPro09}
L.~Ba\v{n}as, Z.~Brze{\'z}niak, A.~Prohl, and M.~Neklyudov.
\newblock A convergent finite-element-based discretization of the stochastic
  {L}andau--{L}ifshitz--{G}ilbert equation.
\newblock {\em IMA Journal of Numerical Analysis},  {\bf } (2013).

\bibitem{Bill99}
P.~Billingsley.
\newblock {\em Convergence of {P}robability {M}easures}.
\newblock Wiley Series in Probability and Statistics: Probability and
  Statistics. John Wiley \& Sons Inc., New York, second edition, 1999.
\newblock A Wiley-Interscience Publication.

\bibitem{BrzGolJer12}
Z.~Brze{\'z}niak, B.~Goldys, and T.~Jegaraj.
\newblock Weak solutions of a stochastic {L}andau--{L}ifshitz--{G}ilbert
  equation.
\newblock {\em Applied Mathematics Research eXpress},  {\bf } (2012), 1--33.

\bibitem{Chen05}
Z.~Chen.
\newblock {\em Finite {E}lement {M}ethods and {T}heir {A}pplications}.
\newblock Scientific Computation. Springer-Verlag, Berlin, 2005.

\bibitem{Cimrak_survey}
I.~Cimr{\'a}k.
\newblock A survey on the numerics and computations for the {L}andau-{L}ifshitz
  equation of micromagnetism.
\newblock {\em Arch. Comput. Methods Eng.},  {\bf 15} (2008), 277--309.

\bibitem{DaPrato92}
G.~Da~Prato and J.~Zabczyk.
\newblock {\em Stochastic {E}quations in {I}nfinite {D}imensions}, volume~44 of
  {\em Encyclopedia of Mathematics and its Applications}.
\newblock Cambridge University Press, Cambridge, 1992.

\bibitem{Flan95}
F.~Flandoli and D.~Gatarek.
\newblock Martingale and stationary solutions for stochastic {N}avier--{S}tokes
  equations.
\newblock {\em Probability Theory and Related Fields},  {\bf 102} (1995),
  367--391.

\bibitem{Gil55}
T.~Gilbert.
\newblock A {L}agrangian formulation of the gyromagnetic equation of the
  magnetic field.
\newblock {\em Phys Rev},  {\bf 100} (1955), 1243--1255.

\bibitem{GuoPu09}
B.~Guo and X.~Pu.
\newblock Stochastic {L}andau--{L}ifshitz equation.
\newblock {\em Differential and Integral Equations},  {\bf 22} (2009),
  251--274.

\bibitem{Johnson87}
C.~Johnson.
\newblock {\em Numerical {S}olution of {P}artial {D}ifferential {E}quations by
  the {F}inite {E}lement {M}ethod}.
\newblock Cambridge University Press, Cambridge, 1987.

\bibitem{LL35}
L.~Landau and E.~Lifschitz.
\newblock On the theory of the dispersion of magnetic permeability in
  ferromagnetic bodies.
\newblock {\em Phys Z Sowjetunion},  {\bf 8} (1935), 153--168.

\bibitem{LMDT13}
K.-N. Le, M.~Page, D.~Praetorius, and T.~Tran.
\newblock On a decoupled linear {FEM} integrator for {E}ddy-current-{LLG}.
\newblock {\em Applicable Analysis},  {\bf } (2014).

\bibitem{LeTra12}
K.-N. Le and T.~Tran.
\newblock A convergent finite element approximation for the quasi-static
  {M}axwell--{L}andau--{L}ifshitz--{G}ilbert equations.
\newblock {\em Computers \& Mathematics with Applications},  {\bf 66} (2013),
  1389 -- 1402.

\end{thebibliography}
\end{document}